\documentclass[a4paper]{amsart}

\usepackage[lite]{amsrefs}

\usepackage{amssymb}              
\usepackage{lmodern,bm}              
\usepackage{stmaryrd}
\usepackage{cancel}
\usepackage{mathrsfs}
\usepackage{longtable}
\usepackage{enumitem}

\makeatletter
\DeclareFontEncoding{LS1}{}{}
\DeclareFontSubstitution{LS1}{stix}{m}{n}
\DeclareMathAlphabet{\skr}{LS1}{stixscr}{m}{n}
\makeatother

\usepackage[all]{xy}

\usepackage{tikz, pgfplots, pgfplotstable}
\usetikzlibrary{calc}
\usetikzlibrary{decorations.markings}
\usetikzlibrary{positioning,arrows}
\usepackage{longtable}

\usepackage{stackengine}

\CompileMatrices

\newtheorem{theorem}{Theorem}[section]

\newtheorem{lemma}[theorem]{Lemma}
\newtheorem{proposition}[theorem]{Proposition}
\newtheorem{corollary}[theorem]{Corollary}

\theoremstyle{definition}

\newtheorem{definition}[theorem]{Definition}
\newtheorem{construction}[theorem]{Construction}

\newtheorem{example}[theorem]{Example}
\newtheorem{remark}[theorem]{Remark}

\newtheorem{algorithm}[theorem]{Algorithm}

\numberwithin{equation}{theorem}

\def\vector2#1#2{\left(\begin{array}{c} #1 \\ #2 \end{array}\right)}
\def\Cl{{\rm Cl}}

\def\KK{{\mathbb K}}
\def\TT{{\mathbb T}}
\def\ZZ{{\mathbb Z}}

\def\QQ{{\mathbb Q}}
\def\PP{{\mathbb P}}

\def\grad{{\rm grad}}

\def\Chi{{\mathbb X}}

\def\div{{\rm div}}

\def\add{{\rm add}}

\def\rk{{\rm rk}}

\def\bangle#1{{\langle #1 \rangle}}

\def\WDiv{{\rm WDiv}}
\def\PDiv{{\rm PDiv}}

\def\Pic{{\rm Pic}}

\def\cone{{\rm cone}}

\def\lcm{{\rm lcm}}

\DeclareMathOperator{\flip}{\mathrm{flip}}
\DeclareMathOperator{\vt}{\mathrm{vt}}
\DeclareMathOperator{\addd}{\mathrm{add}}

\title[Del Pezzo surfaces of Picard number one admitting a torus action]{Del Pezzo surfaces of Picard number one \\ admitting a torus action}

\author[Daniel H\"attig, Beatrice Hafner, J\"urgen Hausen, Justus Springer]{Daniel H\"attig, Beatrice Hafner, J\"urgen Hausen, Justus Springer}

\address{Mathematisches Institut, Universit\"at T\"ubingen,
Auf der Morgenstelle 10, 72076 T\"ubingen, Germany}
\email{daniel.haettig@uni-tuebingen.de}

\address{Mathematisches Institut, Universit\"at T\"ubingen,
Auf der Morgenstelle 10, 72076 T\"ubingen, Germany}
\email{juergen.hausen@uni-tuebingen.de}

\address{Mathematisches Institut, Universit\"at T\"ubingen,
Auf der Morgenstelle 10, 72076 T\"ubingen, Germany}
\email{justus.springer@uni-tuebingen.de}


\begin{document}

\begin{abstract}
We present efficient classification algorithms
for log del Pezzo surfaces with torus action
of Picard number one and given Gorenstein index.
Explicit results are obtained up to Gorenstein
index 200.
\end{abstract}

\maketitle

\section{Introduction}

This article contributes to the classification
of del Pezzo surfaces.
Recall that a \emph{del Pezzo surface} is a normal algebraic
surface $X$ over an algebraically closed field
$\KK$ of characteristic zero that admits an ample
anticanonical divisor $-\mathcal{K}_X$.
The \emph{smooth} del Pezzo surfaces are well
known: the product $\PP_1 \times \PP_1$ of the
projective line with itself, the projective
plane~$\PP_2$ and the blowing-ups of~$\PP_2$ in
up to eight points in general position.
For the \emph{singular} del Pezzo surfaces,
we need to impose suitable conditions on the
singularities in order to end up with any kind
of finiteness features allowing a classification
comparable to the smooth case.

In the singular case, it is common to restrict
to~\emph{log del Pezzo surfaces} $X$,
which means that all discrepancies of some resolution of
singularities $X' \to X$ are greater than $-1$.
Log del Pezzo surfaces are necessarily
rational~\cite[Prop.~3.6]{Nak}
but still form a huge class without suitable
boundedness features.
A common strategy is to filter by the \emph{Gorenstein index},
that means the smallest positive integer $\iota_X$ such
that $\iota_X \mathcal{K}_X$ is a Cartier divisor.
The simplest case, $\iota_X = 1$, gives
the \emph{Gorenstein del Pezzo surfaces $X$}
which have been classified by Hidaka/Watanabe~\cite{HiWa}.
Moreover, Alexeev/Nikulin~\cite{AlNi} and Nakayama~\cite{Nak}
succeeded in classifying the log del Pezzo surfaces of
Gorenstein index two.
Nakayama's approach was extended by Fujita/Yasutake~\cite{FuYa}
to treat the case of Gorenstein index three and also
to provide a strategy to investigate higher Gorenstein
indices.

The situation becomes much more accessible and explicit
if one considers
del Pezzo surfaces~$X$ coming with a (non-trivial) torus
action.
In this setting, the most symmetric ones are
the \emph{toric del Pezzo surfaces}, that means
those with an effective action of the two-dimensional
torus $\KK^* \times \KK^*$; these are automatically
log del Pezzo.
Kasprzyk, Kreuzer and Nill~\cite{KaKrNi} provide
us with a classification of the toric del Pezzo
surfaces up to Gorenstein index~16.
The other possible case is given by the
\emph{non-toric del Pezzo $\KK^*$-surfaces},
that means those allowing only an effective action
of a one-dimensional torus $\KK^*$.
In this case, complete classifications exist up to
Gorenstein index~3; where Huggenberger~\cite{Hug}
treated the Gorenstein case, in the Gorenstein
indices 2 and 3, the case of Picard number one
has been settled by~S\"u\ss~\cite{Su} and
the cases of higher Picard numbers can be
found in~\cite[Cor.~1.2]{Hae} .

In the present article, we focus on log del Pezzo
surfaces of Picard number one coming with a torus
action.
In this setting, the toric ones are precisely the
\emph{fake weighted projective planes}.
These and their higher dimensional analogues,
the \emph{fake weighted projective spaces}
have been studied by several authors~\cite{Bae,CGKN,Ka}.
In our classification, we benefit from a close
connection to decompositions of $1/\iota$,
where~$\iota$ stands for the Gorenstein index,
into a sum of three unit fractions, which
finally leads to our Classification
Algorithm~\ref{alg:classfwpp}. Up to isomorphism,
the algorithm delivers 117.065 toric log del
Pezzo surfaces of Picard number one and Gorenstein
index at most 200.
In the non-toric case, the $\KK^*$-surfaces with
at most cyclic quotient singularities form
the richest case. Here we use again a connection to
unit fractions in the corresponding
Classification Algorithm~\ref{alg:classkstar}.
The cases admitting more serious singularities
turn out to be less productive and can be directly
addressed via tha bounds provided in
Proposition~\ref{prop:non-qs-bounds}.
We obtain 154.138 families of non-toric
log del Pezzo $\KK^*$-surfaces of Picard number
one and Gorenstein index at most 200.

Let us give a summarizing impression of
our classification results. We refer to
Propositions~\ref{prop:toric-numbers}
and~\ref{prop:kstar-numbers} for more
details.
Moreover, all resulting data can be found
at~\cite{TDB}.

\begin{theorem}
There are 271.203 families of log del Pezzo surfaces
with torus action of Picard number one and Gorenstein
index at most 200.
The numbers of families for given Gorenstein
index develop as follows:

\bigskip

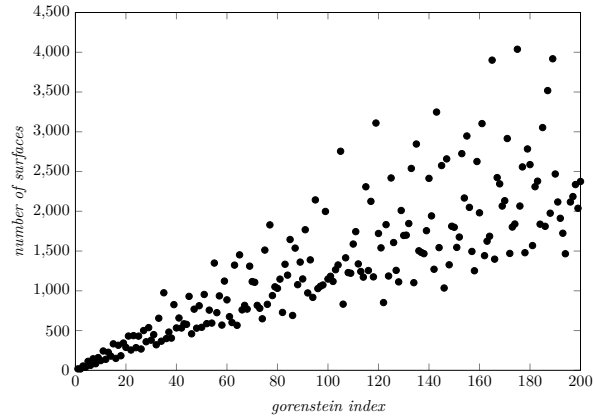
\begin{center}

\begin{tikzpicture}[scale=0.6]
\begin{axis}[
    xmin = 0, xmax = 200,
    ymin = 0, ymax = 4500,
    width = \textwidth,
    height = 0.75\textwidth,
    xtick distance = 20,
    ytick distance = 500,
    xlabel = {gorenstein index},
    ylabel = {number of surfaces}
]
\addplot[only marks] table {combinedTable.txt};
\end{axis}
\end{tikzpicture}

\end{center}

\medskip

\end{theorem}

The computational treatment of the log del Pezzo
surfaces $X$ with torus action is made possible by
an encoding of the surfaces in terms of certain
integral matrices~$P$.
If $X$ is acted on effectively by a two-dimensional
torus, then it is a toric surface, and~$P$ is the
$2 \times 3$ matrix having the primitive generators
of the describing fan of~$X$ as its columns.
If $X$ only allows an effective action a one-dimensional
torus, then it can be realized in a very specific way
as a subvariety of a toric variety $Z$ of higher
dimension and the matrix $P$ encoding $X \subseteq Z$
has the primitive generators of the describing fan of
$Z$ as its columns.
These approaches rely on the general theories of toric
varieties~\cite{CoLiSc,Dan} and rational varieties with
a torus action of complexity one initiated in~\cite{HaHe,HaSu}.
Restricting to Picard number one means that the ambient
toric varieties $Z$ showing up are all
fake weighted projective spaces.
Section~\ref{sec:basics-fwps} introduces to the
latter ones in terms of basic algebraic geometry with
clear interfaces to toric geometry when using methods
from there.
In Section~\ref{sec:classify-fwps}, we present
and prove the classification algorithm for the
log del Pezzo surfaces of Picard number one
admitting an effective action
of a two-dimensional torus.
Sections~\ref{sec:basics-kstar} and~\ref{sec:geom-kstar}
serve to develop the necessary parts of the theory on 
rational $\KK^*$-surfaces.
Adapting to the case of Picard number one allows us
to stay in terms of basic algebraic geometry, giving
clear interfaces to the general theory when necessary.
In Section~\ref{sec:classify-kstar}, we
present our classification procedure for the log del
Pezzo $\KK^*$-surfaces of Picard number one.
and Section~\ref{sec:tables} provides tables on the
classification results.

\goodbreak

\tableofcontents

\section{Basics on fake weighted projective spaces}
\label{sec:basics-fwps}

The projective space $\PP_n$ is the set of
all lines through the origin in the affine
$(n+1)$-plane.
We may regard $\PP_n$ as well as the quotient
of the pointed affine $(n+1)$-plane by 
the one-dimensional torus $\KK^*$
acting via scalar multiplication.
Allowing more generally
\emph{one-dimensional quasitorus actions},
this point of view brings us to the
\emph{fake weighted projective spaces}.
We give a basic introduction, define all necessary
notions, present two possible constructions of fake
weighted projective spaces in detail and show
how to turn fake weighted projective spaces into
toric varieties.

A \emph{torus} is an algebraic group
isomorphic to some \emph{standard $n$-torus}
$\TT^n = (\KK^*)^n$.
We denote by $C(m) \subseteq \KK^*$ the
group of $m$-th roots of unity.
A \emph{quasitorus} is an algebraic group
isomorphic to a \emph{standard $n$-quasitorus},
that means a direct product 
$$
\TT^n \times C,
\qquad
C \ = \ C(m_1) \times \ldots \times C(m_k).
$$
A \emph{character} of a quasitorus $H$ is a
homomorphism $\chi \colon H \to \KK^*$
of algebraic groups.
Explicitly, given a standard quasitorus $\TT^n \times C$,
set
$\Gamma := \ZZ/m_1\ZZ \times \ldots \times \ZZ/m_k\ZZ$.
Then every $\omega = (w,\eta) \in \ZZ^n \times \Gamma$ 
defines a character
$$
\chi^\omega \colon \TT^n \times C \to \KK^*,
\quad
(s,\zeta) \mapsto s^w\zeta^\eta,
\qquad
s^{w} := s_1^{w_1}\cdots s_n^{w_n},
\quad
\zeta^{\eta}
:=
\zeta_1^{\eta_1} \cdots \zeta_k^{\eta_k},
$$
and this assignment sets up an isomorphism
$\ZZ^n \times \Gamma \to \Chi(\TT^n \times C)$
onto the \emph{character group} of
$\TT^n \times C$. We are ready to present
our first concrete construction of fake weighted
projective spaces.

\begin{construction}[Fake weighted projective spaces]
\label{constr:fwps1}
Consider a one-dimensional quasitorus, given as
a direct product
$$
H \ = \ \KK^* \times C,
\qquad
C \ = \ C(m_1) \times \ldots \times C(m_k)
$$
and characters $\chi^{\omega_0}, \ldots, \chi^{\omega_n}$,
where $\omega_i = (w_i,\zeta_i) \in \ZZ \times \Gamma$
with $w_0, \ldots, w_n > 0$.
Then we have an action
$$
H \times \KK^{n+1} \ \to \ \KK^{n+1},
\qquad
h \cdot z
\ := \
(
\chi^{\omega_0}(h)z_0, \ldots, \chi^{\omega_n}(h)z_n
)
$$
such that $0 \in \KK^{n+1}$ lies in the closure of
every $H$-orbit.
The \emph{fake weighted projective space} associated
with $\omega_0, \ldots, \omega_n$ is the orbit space
$$
\PP(\omega_0, \ldots, \omega_n)
\ := \
(\KK^{n+1} \setminus \{0\}) / H.
$$
\end{construction}

We gather basic properties. Given a variety $Z$
acted on by a group $G$, a morphism $Z \to X$ of
varieties  is called \emph{$G$-invariant} if it
is constant along the  $G$-orbits.

\begin{proposition}
Consider a fake weighted projective space
$\PP(\omega_0, \ldots, \omega_n)$ resulting
from Construction~\ref{constr:fwps1}
and the canonical map
$$
\pi \colon \KK^{n+1} \setminus \{0\} \ \to \ \PP(\omega_0, \ldots, \omega_n),
\qquad
z \ \mapsto \ H \cdot z.
$$
Then $\PP(\omega_0, \ldots, \omega_n)$ is
an irreducible, normal, projective variety
of dimension~$n$ once we installed the
quotient topology w.r.t.~$\pi$ and the
structure sheaf
$$
\mathcal{O}(U)
\ := \
\{f \colon U \to \KK; \ f \circ \pi \in \mathcal{O}(\pi^{-1}(U))\}
\ = \
\mathcal{O}(\pi^{-1}(U))^H .
$$
Moreover,
$\pi \colon \KK^{n+1} \setminus \{0\} \to \PP(\omega_0, \ldots, \omega_n)$
is an affine $H$-invariant morphism and
every $H$-invariant morphism
$\varphi \colon \KK^{n+1} \setminus \{0\} \to X$
uniquely factors as 
$$
\xymatrix{
{\KK^{n+1} \setminus \{0\}}
\ar[rr]^{\varphi}
\ar[dr]_{\pi}
&&
X
\\
&
\PP(\omega_0, \ldots, \omega_n)
\ar[ur]_{\bar \varphi}
&
}
$$
\end{proposition}

\begin{proof}
All this is a consequence of the fact that
$\PP(\omega_0, \ldots, \omega_n)$ is a
GIT-quotient of the action of $H$ on $\KK^{n+1}$
in the sense of~\cite{Mu}.
More explicitly, the $H$-action on $\KK^{n+1}$
stems from the $(\ZZ \times \Gamma)$-grading
of $\KK[T_0,\ldots,T_n]$ given by
$\deg(T_i) = \omega_i$. Thus, projectivity of
$\PP(\omega_0, \ldots, \omega_n)$ and the
universal property of~$\pi$ follow
from \cite[Prop.~3.1.2.2]{ArDeHaLa}.
Finally, using normality of $\KK^{n+1}$ and the 
universal property of~$\pi$, we obtain normality 
of $\PP(\omega_0, \ldots, \omega_n)$.
\end{proof}

Note that for the special case $H = \KK^*$,
Construction~\ref{constr:fwps1} delivers
precisely the \emph{weighted projective spaces}.
Moreover, it shows that every fake weighted
projective space is covered by a weighted
one.

\begin{remark}
\label{rem:fwpscover}
Every fake weighted projective space
is a quotient of a weighted projective
space by a finite abelian group:
there is a commutative diagram
$$
\xymatrix@R=10pt{
&
{\KK^{n+1} \setminus \{0\} \quad }
\ar[dl]_{/\KK^*}
\ar[dr]^{/H}
\\
{\PP(w_0, \ldots, w_n)}
\ar[rr]_{/C}
&
&
{\PP(\omega_0, \ldots, \omega_n)}
}
$$
where $H = \KK^* \times C$ and $\omega_i = (w_i,\eta_i)$
are as in Construction~\ref{constr:fwps1}
and we have the induced action of $C$ on the weighted
projective space $\PP(w_0,\ldots,w_n)$.
\end{remark}

\begin{example}
\label{ex:fwpp-ex-1}
Consider the one-dimensional quasitorus $H = \KK^* \times C(4)$
and the action of $H$ on $\KK^3$ given by
$$
(s,\zeta) \cdot (z_0,z_1,z_2)
\ = \
(s^2 \zeta z_0, s\zeta^2 z_1,s \zeta z_2).
$$
Then we have $\Gamma := \ZZ/ 4 \ZZ$ and, denoting
the elements of $\Gamma$ by $\bar 0$, $\bar 1$, $\bar 2$,
$\bar 3$, the $\omega_i = (w_i,\eta_i)$ are explicitly
given as  
$$ 
\omega_0 = (2, \bar 1),
\quad
\omega_1 = (1, \bar 2)
\quad
\omega_2 = (1, \bar 1).
$$
We arrive at a fake weighted projective plane
$\PP(\omega_0,\omega_1,\omega_2) = (\KK^3 \setminus\{0\})/H$,
which by Remark~\ref{rem:fwpscover} 
comes with a 4:1 cover by the weighted projective plane
$\PP(2,1,1)$.
\end{example}

Many of the well known concepts around the projective space
directly generalize to the fake weighted
projective spaces. For instance, we can introduce analogues
of homogeneous coordinates, coordinate hyperplanes
and the standard affine charts as follows.

\begin{definition}
Consider a fake weighted projective space $\PP(\omega_0,\ldots,\omega_n)$
arising from Construction~\ref{constr:fwps1}.
Given $z = (z_0,\ldots,z_r)$ in $\KK^{n+1} \setminus \{0\}$, set
$$
[z]
\ = \
[z_0,\ldots,z_n]
\ :=  \ 
H \cdot z
\ \in \
\PP(\omega_0,\ldots,\omega_n).
$$
Then we call $[z]$ a presentation of $H \cdot z$ in
\emph{homogeneous coordinates}.
For any two points $z, z' \in \KK^{n+1} \setminus \{0\}$,
we have
$$
[z]  = [z']
\ \Leftrightarrow \
z' = h \cdot z \text{ for some } h \in H.
$$
Moreover, for $k = 0, \ldots, n$ we define the 
\emph{$k$-th coordinate divisor} to be
the closed, irreducible $(n-1)$-dimensional
subvariety
$$
D_k
\ := \
\{[z]; \ z \in \KK^{n+1} \setminus \{0\}, \ z_k = 0\}
\ \subseteq \
\PP(\omega_0,\ldots,\omega_n).
$$
Finally, for $k = 0, \ldots, n$, the 
\emph{$k$-th affine chart} is the open affine
subvariety obtained by removing the $k$-th
coordinate divisor:
$$
Z_k
\ := \ 
\PP(\omega_0,\ldots,\omega_n) \setminus D_k.
$$
\end{definition}

\begin{remark}
Every coordinate divisor of a fake weighted projective
space is itself a fake weighted projective space.
Moreover, intersecting coordinate divisors, we obtain
\emph{coordinate subspaces}, which again are fake 
weighted projective spaces.
\end{remark}

An important feature of fake weighted projective
spaces is that they are examples of toric varieties;
recall that a \emph{toric variety} is a normal 
variety $X$ together with an effective action
of a torus $\TT$ having an open orbit
$\TT \cdot x_0 \subseteq X$.
The torus $\TT$ is called the \emph{acting torus}
of $X$.

\begin{remark}
\label{rem:fwpp2tv1}
Consider a fake weighted projective space
$\PP(\omega_0, \ldots, \omega_n)$
produced by Construction~\ref{constr:fwps1}.
Then the torus $\TT^{n+1} = (\KK^*)^{n+1}$ acts on
$\KK^{n+1} \setminus \{0\}$ via
$$
t \cdot z
\ = \
(t_0z_0, \ldots, t_nz_n).
$$
This action commutes with the $H$-action
and induces an effective almost transitive
action of the torus $\TT^n \cong \TT^{n+1}/H$ 
on $\PP(\omega_0, \ldots, \omega_n)$,
turning it into a toric variety.
\end{remark}

In order to benefit from the rich theory of
toric varieties~\cite{Dan,Ful,CoLiSc}, we
provide another way to construct fake weighted
projective spaces, starting with a certain integral
matrix as input data.
Besides strengthening the connection to toric varieties,
this approach also yields an appropriate encoding for
our subsequent computational considerations.

\begin{construction}
[Fake weighted projective spaces via integral matrices]
\label{constr:fwps2}
Consider an integral $n \times (n+1)$ matrix
$$
P \ = \ [v_0,\ldots, v_n],
$$
the columns $v_0, \ldots, v_n$ of which are 
primitive vectors in $\ZZ^n$ generating~$\QQ^n$
as a convex cone.
The matrix $P=(p_{ij})$ defines a homomorphism
of tori
$$
p \colon \TT^{n+1} \ \to \ \TT^n,
\qquad
t \ \mapsto \ (t^{P_{0*}}, \ldots, t^{P_{n*}}),
$$
where $t^{P_{i*}} = t_0^{p_{i0}} \cdots t_n^{p_{in}}$
is the monomial in $t_0,\ldots,t_n$ having the
$i$-th row of $P$ as its exponent vector.
This in turn gives us a one-dimensional quasitorus
$$
H \ := \ \ker(p) \ \subseteq \ \TT^{n+1}.
$$
Denote by $\chi^{\omega_i} \in \Chi(H)$ the character 
obtained by restricting the $i$-th coordinate function
$\TT^{n+1} \to \KK^*$.
Then the subgroup $H \subseteq \TT^{n+1}$ acts on $\KK^{n+1}$
via
$$
h \cdot z
\ = \
(\chi^{\omega_0}(h)z_0, \ldots, \chi^{\omega_n}(h)z_n).
$$ 
Suitably splitting
$H = \KK^* \times C$ with
$C = C(1) \times \ldots \times C(k)$,
we can write $\omega_i =(w_i,\eta_i)$
as in Construction~\ref{constr:fwps1}
and arrive at a fake weighted projective
space
$$
Z(P)
\ := \
(\KK^{n+1} \setminus \{0\}) / H
\ = \ 
\PP(\omega_0, \ldots, \omega_n).
$$
\end{construction}

\begin{example}
Let us see how to obtain the fake weighted projective
plane from Example~\ref{ex:fwpp-ex-1} by means of
Construction~\ref{constr:fwps2}.
Consider the matrix
$$
P
\ = \
\left[
\begin{array}{ccc}
1 & 1 & -3 
\\
0 & 4 & -4
\end{array}
\right].
$$
Note that the columns of $P$ are primitive and
generate $\QQ^2$ as a cone.
The homomorphism of tori associated with $P$
is given as 
$$
p \colon \TT^3 \ \to \ \TT^2,
\qquad\qquad
(t_0,t_1,t_2)
\ \mapsto \
\left(\frac{t_0t_1}{t_2^3}, \, \frac{t_1^4}{t_2^4}\right).
$$
One directly computes
$$
H
\ := \
\ker(p)
\ = \ 
\{(s^2\zeta, s\zeta^2 ,s\zeta); \ s \in \KK^*, \ \zeta \in C(4)\}
\ \subseteq \
\TT^3.
$$
Moreover, $(s,\zeta) \mapsto (s^2\zeta, s\zeta^2 ,s\zeta)$
yields a splitting $\KK^* \times C(4) \cong H$.
Thus, we indeed arrive at Example~\ref{ex:fwpp-ex-1} again:
$$
Z(P)
 = 
(\KK^3 \setminus \{0\})/H
 = 
\PP(\omega_0,\omega_1,\omega_2),
\quad
\omega_0 = (2, \bar 1),
\
\omega_1 = (1, \bar 2),
\
\omega_2 = (1, \bar 1).
$$
\end{example}

Observe that Construction~\ref{constr:fwps2} is 
more special than Construction~\ref{constr:fwps1}
in the sense that it will not produce all
character lists that will appear within
Construction~\ref{constr:fwps1},
as for instance $H=\KK^*$, $n=1$ and $k=0$
with $(\omega_0,\omega_1) = (w_0,w_1) = (2,2)$.
More precisely, we can say the following.

\begin{remark}
\label{rem:well-formed}
In Construction~\ref{constr:fwps2}, the matrix $P$ 
has pairwise distinct primitive columns.
This merely means that any $n$ members 
of the resulting weight list
$(\omega_0,\ldots,\omega_n)$ generate
$\ZZ \times \Gamma$ as a group;
use~\cite[Lemma~2.1.4.1]{ArDeHaLa}.
In particular, the associated list
$(w_0,\ldots,w_n)$ of the $\ZZ$-parts of
the $\omega_i$ is \emph{well-formed} in the
sense that any $n$ of the $w_i$ are coprime.
\end{remark}

Remark~\ref{rem:fwpp2tv1} tells us that
any fake weighted projective space is a
toric variety.
For the $Z(P)$ arising from Construction~\ref{constr:fwps2},
we will use the homomorphism $p \colon \TT^{n+1} \to \TT^n$
given by the matrix $P$ in order to
understand the torus action more concretely.

\begin{remark}
\label{rem:fwpp2tv2}
Consider a fake weighted projective space
$Z(P)$ arising from Construction~\ref{constr:fwps2}
with the action of the torus $\TT^{n+1}/H$ as
provided by Remark~\ref{rem:fwpp2tv1}.
Then we obtain a commutative diagram 
$$
\xymatrix{
{\TT^{n+1}}
\ar@{}[r]|{\subseteq\quad}
\ar[d]_p
&
{\quad \KK^{n+1} \setminus \{0\}}
\ar[d]^{\pi}
\\
{\TT^n}
\ar[r]
&
Z(P) ,
}
$$
where the lower arrow is an isomorphism
from the torus $\TT^n$ onto the torus $\TT^{n+1}/H$
which in turn equals the open set
$\pi(\TT^{n+1}) \subseteq Z(P)$.
This allows us to regard $\TT^n$ as the acting torus
of the toric variety $Z(P)$.
\end{remark}

\begin{proposition}
\label{prop:fwpstoric}
Consider a fake weighted projective space $Z = Z(P)$
provided by Construction~\ref{constr:fwps2}.
Then the quasitorus $H$ acts freely on the open
set
$$
\bigcup_{0 \le j < k \le n} \KK^{n+1} \setminus V(T_j,T_k)
\ \subseteq \
\KK^{n+1}.
$$
The isotropy group of $\TT^n$ at any point $[z] \in D_k$
with $z_j \ne 0$ for $j \ne k$ is given in terms of the
$k$-th column $v_k = (v_{k1}, \ldots, v_{kn})$ as 
$$
\TT^n_{[z]}
\ = \
\{(s^{v_{k1}}, \ldots, s^{v_{kn}}); \ s \in \KK^* \}
\ \subseteq \
\TT^n.
$$
The fixed points of the $\TT^n$-action on $Z$ are
precisely the points $\mathbf{z}(k) \in Z$
having all homogeneous coordinates except the $k$-th
one equal to zero:
$$ 
\mathbf{z}(0) \ = \ [1,0,\ldots,0],
\quad
\mathbf{z}(1) \ = \ [0,1,0, \ldots,0],
\quad
\ldots,
\quad
\mathbf{z}(n) \ = \ [0,\ldots,0,1].
$$
The affine chart $Z_k$ is the minimal $\TT^n$-invariant 
open set containing the point $\mathbf{z}(k)$.
Moreover, identifying $\TT^n$ with its open 
orbit in $Z$, we have
$$
Z_0 \cap \ldots \cap Z_n
\ = \ 
\TT^n
\ = \ 
Z \setminus (D_0 \cup \ldots \cup D_n).
$$
\end{proposition}

\begin{proof}
Let us see why $H$ acts freely on the set of points
$z = (z_0,\ldots,z_n)$ with at most one coordinate
equal to zero.
For instance, consider
$z = (0,z_1,\ldots,z_n)$.
We have to show that the isotropy group $H_z$
is trivial.
For any $h \in H$, we have
$$
h \cdot z  = z
\quad \iff \quad
\chi^{\omega_1}(h) = \ldots = \chi^{\omega_n}(h) = 1.
$$
Remark~\ref{rem:well-formed} tells us 
that $\omega_1,\ldots,\omega_n$ generate
$\ZZ \times \Gamma$ as a group.
Thus, $\chi^{\omega_1}, \ldots , \chi^{\omega_n}$
generate the character group $\Chi(H)$,
hence $\mathcal{O}(H)$ and thus
separate the points of~$H$.
Consequently, $h \cdot z  = z$ only
happens for $h=1$.

We turn to the statements involving isotropy groups
of the $\TT^n$-action on $Z$.
First, consider any point $[z]$ of $Z$.
Then, in terms of the homomorphism
$p \colon \TT^{n+1} \to \TT^n$ defined by the matrix
$P$ and due to $H = \ker(p)$ we have
$$
\TT^n_{[z]}
\ = \
\{p(t); \ t \in \TT^{n+1}, \ t \cdot z \in H \cdot z\}
\ = \
p(\TT^{n+1}_z).
$$

Now look at a point $z \in \KK^{n+1}$ having exactly one
coordinate equal to zero, say $z = (0,z_1, \ldots, z_n)$.
Then the isotropy group $\TT^{n+1}_z$ consists precisely
of the elements $(s,1,\ldots,1) \in \TT^{n+1}$,
where $s \in \KK^*$.
These are mapped via $p \colon \TT^{n+1} \to \TT^n$
precisely to the elements 
$(s^{v_{11}},\ldots,s^{v_{1n}}) \in \TT^n$,
where $s \in \KK^*$.

Next consider $\mathbf{z}(0), \ldots, \mathbf{z}(n) \in Z$.
Since the $\omega_i = (w_i,\eta_i)$ satisfy $w_i > 0$ for
$i = 0, \ldots, n$, the above formula yields that
each $\mathbf{z}(i)$ has $\TT^n$ as its isotropy group.
Moreover, any $z \in \KK^{n+1}$ with two or more non-zero
coordinates has a $\TT^{n+1}$-orbit of dimension
at least two.
Using the above formula once more, we see that the $[z]$
has a non-trivial $\TT^n$-orbit.

Finally, the $D_k \subseteq Z$ consist by definition
of all points $[z] \in Z$ with $z_k = 0$ and we have
$Z = Z_k \cup D_k$ as a disjoint union.
Moreover, Remark~\ref{rem:fwpp2tv2} shows
that $\TT^n \subseteq Z$ consists precisely
of the points with only non-zero coordinates.
Altogether, this gives the last statement of the
proposition.
\end{proof}

The $n$-dimensional toric varieties are in 
correspondence with \emph{fans in $\ZZ^n$}, 
that means finite sets $\Sigma$ 
of pointed, polyhedral, convex cones in 
$\QQ^n$ such that for any cone of $\Sigma$ 
all its faces belong to $\Sigma$ as well 
and any two cones of $\Sigma$ intersect
in a common face.
Let us see how to detect the defining fan
of a fake weighted projective space $Z$
given by Construction~\ref{constr:fwps2}
and, as a direct consequence, its 
\emph{divisor class group} $\Cl(Z)$ and
\emph{Cox ring}
$$
\mathcal{R}(Z)
\ = \
\bigoplus_{\Cl(Z)} \Gamma(Z,\mathcal{O}(D)).
$$
A reference on Cox rings of toric varieties 
is~\cite{Co}; see
also~\cite[Sections~2.1.3 and~2.1.4]{ArDeHaLa}
and~\cite[Chap.~5]{CoLiSc} for more details
on Cox rings and Cox's quotient construction.
However, a deeper understanding of the
theory around Cox rings is not needed throughout
this text.

\begin{proposition}
\label{prop:fwpp2tv}
Let $P = [v_0,\ldots,v_n]$ be as in
Construction~\ref{constr:fwps2}.
Then $Z = Z(P)$ is the toric variety 
associated with the fan $\Sigma$ in $\ZZ^n$
given by
$$
\Sigma
\ = \
\{
\cone(v_{i_1}, \ldots, v_{i_m}); \
0 \le i_1 < \ldots < i_m \le n, \
m \le n
\}.
$$
With $K := \ZZ^{n+1}/P^*\ZZ^n$,
we obtain the character
group of the quasitorus $H$
and the divisor class group of the
fake weighted projective space $Z$ as 
$$
\Chi(H)
\ \cong \
K 
\ \cong \
\Cl(Z).
$$
With the canonical projection $Q \colon \ZZ^{n+1} \to K$,
these ismorphisms allow to relate the characters
$\chi^{\omega_i}$ and the coordinate divisors
$D_i$ to each other via
$$
\omega_i
\ = \
Q(e_i)
\ = \
[D_i],
\quad
i = 0, \ldots, n.
$$
Moreover, the Cox ring of the toric variety $Z = Z(P)$
equals its $K$-homogeneous coordinate ring and
thus is given as the $K$-graded polynomial ring
$$
\mathcal{R}(Z)
\ = \
\KK[T_0,\ldots,T_n],
\qquad
\deg(T_i)
\ = \
Q(e_i)
\ \in \
K.
$$
Finally Cox's quotient presentation of
the toric variety $Z = Z(P)$ is precisely 
the quotient map showing up in its construction:
$$
\pi \colon 
\KK^{n+1} \setminus \{0\} \ \to \ Z = (\KK^{n+1} \setminus \{0\})/H,
\qquad
z \ \mapsto \ [z] := H \cdot z.
$$
\end{proposition}

\begin{proof}
In the set $\Sigma$, the cones   
$\sigma_i := \cone(v_k; \ k \ne i)$
are maximal with respect to inclusion.
Since $v_0,\ldots,v_n$ generate
$\QQ^n$ as a cone,
each $\sigma_i$ is $n$-dimensional,
hence pointed, and
any two of them intersect in
a common facet.
Thus, $\Sigma$ is a fan and there is
an associated toric variety
$Z(\Sigma)$.
Cox's quotient presentation~\cite[Sec.~5.1]{CoLiSc}
and~\cite[Sec.~2.1.3]{ArDeHaLa}
reproduces $Z(\Sigma)$ as the quotient
of $\KK^{n+1} \setminus \{0\}$
by $H = \ker(P) \subseteq \TT^{n+1}$,
acting exactly as in
Construction~\ref{constr:fwps2}.
We conclude $Z(\Sigma) = Z(P)$.
For the remaining statements,
we refer to~\cite[Constr.~2.1.3.1]{ArDeHaLa}.
\end{proof}

\begin{proposition}
\label{prop:allfwpsviaP}
Let $Z$ be a fake weighted projective space.
Then $Z \cong Z(P)$ holds with~$Z(P)$
arising from Construction~\ref{constr:fwps2}.
\end{proposition}

\begin{proof}
By Remark~\ref{rem:fwpp2tv1} any fake weighted
projective space $Z$ is a toric variety
and thus $Z \cong Z(\Sigma)$ with a toric variety
$Z(\Sigma)$ given by a fan $\Sigma$ in~$\ZZ^n$.
Applying Cox's quotient presentation~\cite[Sec.~5.1]{CoLiSc}
and~\cite[Sec.~2.1.3]{ArDeHaLa}
to $Z(\Sigma)$, we see that the matrix $P$
having the primitive generators of the
fan~$\Sigma$ as its columns satisfies
$Z(\Sigma) \cong Z(P)$.
\end{proof}

\begin{proposition}
\label{prop:fwpsPequiv}
We have $Z(P) \cong Z(P')$ if and only if
$P' = A \cdot P \cdot S$ holds with a unimodular
matrix $A$ and a permutation matrix $S$.
\end{proposition}

\begin{proof}
First we note that $Z(P)$ and $Z(P')$ are isomorphic
as algebraic varieties if and only if they are
isomorphic as toric varieties; the reason for this
is that toric varieties have a linear algebraic
automorphism group and thus any two maximal tori
of the automorphism group are conjugate.
Next, we remark that $Z(P)$ and~$Z(P')$ are 
isomorphic as toric varieties if and only
if there is a unimodular matrix~$A$ mapping the fan
of $Z(P)$ cone-wise to the fan of $Z(P')$,
which in turn means $P' = A \cdot P \cdot S$
with the matrix $A$ and a permutation
matrix $S$.
\end{proof}

\section{Classifying fake weighted projective planes}
\label{sec:classify-fwps}

We pesent our procedure for efficiently
classifying fake weighted projective planes
of given Gorenstein index.
Recall that a normal variety~$X$ is \emph{$\QQ$-Gorenstein}
if some positive multiple of its canonical
divisor $\mathcal{K}_X$ is Cartier.
The \emph{Gorenstein index} of a $\QQ$-Gorenstein
variety $X$
is the smallest positive integer $\iota_X$
such that $\iota_X \mathcal{K}_X$ is Cartier.
Here is the main result of the section.

\begin{theorem}
There are 117.065 isomorphy classes of toric
del Pezzo surfaces of Picard number
one and Gorenstein
index at most 200.
The numbers of isomorphy classes for given Gorenstein
index develop as follows:

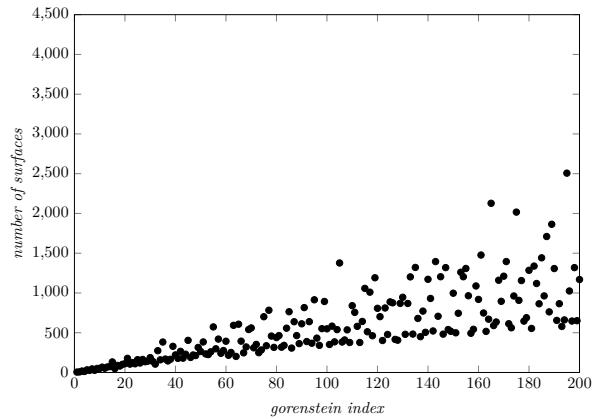
\begin{center}
  
\begin{tikzpicture}[scale=0.6]
\begin{axis}[
    xmin = 0, xmax = 200,
    ymin = 0, ymax = 4500,
    width = \textwidth,
    height = 0.75\textwidth,
    xtick distance = 20,
    ytick distance = 500,
    xlabel = {gorenstein index},
    ylabel = {number of surfaces}
]
\addplot[only marks] table {toricTable.txt};
\end{axis}
\end{tikzpicture}

\end{center}
\end{theorem}

The result is obtained by applying the classification
algorithm~\ref{alg:classfwpp}, developed throghout
this section.
We first provide the necessary facts on
the geometry of fake weighted projective
spaces $Z(P)$, using their structure as toric
varieties;
see~\cite[Prop.~4.1.2, Thm.~4.1.3]{CoLiSc}
and~\cite[Prop.~2.1.2.7]{ArDeHaLa} for the
details.

\begin{proposition}
\label{prop:tordiv}
Consider $Z = Z(P)$ given by Construction~\ref{constr:fwps2}
and a character $\chi^u \in \chi(\TT^n)$.
Then $\chi^u$ is a rational function on $Z$
with divisor
$$
\div(\chi^u) 
\ = \ 
\bangle{u,v_0} D_0
+ \ldots + 
\bangle{u,v_n} D_n.
$$
\end{proposition}

\begin{proposition}
\label{prop:fwpsqfact}
Every fake weighted projective space $Z$ is
$\QQ$-factorial that means that for any Weil
divisor on $Z$ some positive multiple is a
Cartier divisor.
\end{proposition}

\begin{proof}
By Proposition~\ref{prop:allfwpsviaP}
we may assume $Z = Z(P)$.
Since $v_0,\ldots,v_n$ generate~$\QQ^n$ as a cone,
each of the cones of the fan $\Sigma$ is generated
by a linearly independent collection
of the $v_j$.
Now~\cite[Prop.~4.2.7]{CoLiSc} tells us that
$Z$ is $\QQ$-factorial.
\end{proof}

\begin{proposition}
\label{prop:cartier-fwpp}
For a fake weighted projective space
$Z = Z(P)$ and any Weil divisor
$D = a_0D_0 + \ldots + a_nD_n$ on $Z$,
the following statements are equivalent.
\begin{enumerate}
\item
The divisor $D$ is Cartier on a
neighbourhood of $\mathbf{z}(i) \in Z$.
\item
The divisor $D$ is Cartier on the affine
open subvariety $Z_i \subseteq Z$.  
\item
There is a linear form $u \in \ZZ^n$ such
that $D = \div(\chi^u)$ holds on $Z_i \subseteq Z$.
\item
There is a linear form $u \in \ZZ^n$ such
that $\bangle{u,v_j} = a_j$ for all $j \ne i$.
\item
The class 
$a_0\omega_0 + \ldots + a_n \omega_n \in K$
is a multiple of $\omega_i \in K$.
\end{enumerate}
\end{proposition}

\begin{proof}
Assume that~(i) holds. Then $D$ is Cartier on 
an open neighbourhood $U \subseteq Z$ of
$\mathbf{z}(i) \in Z$.  Since $D$ is
$\TT^n$-invariant, $D$ is also Cartier on
$\TT^n \cdot U$. Since $Z_i \subseteq Z$
is the smallest $\TT^n$-invariant open
neighbourhood containing $\mathbf{z}(i)$,
we have $Z_i \subseteq \TT^n \cdot U$ and 
conclude that $D$ is Cartier on $Z_i$.
Now assume that~(ii) holds.
Then~\cite[Prop.~4.2.2]{CoLiSc} tells us
that the restriction of $D$ to $Z_i$ is the
divisor of a character function $\chi^u$.
Assertions~(iii) and (iv) are equivalent
due to Proposition~\ref{prop:tordiv}.
If~(iii) holds, then we have
$D - \div(\chi^u) \ = \ bD_i$
with $b \in \ZZ$.
Finally, if~(v) holds, then the class  
$a_0 \omega_0 + \ldots + a_n \omega_n$ equals 
$b\omega_i$ with some $b \in \ZZ$.
Consequently, $D - bD_i$ is principal on $Z$ 
and hence $D$ is principal on the neighbourhood
$Z_i$ of $\mathbf{z}(i)$.
\end{proof}

For a normal variety $X$ and a point $x \in X$, one
defines the \emph{local class group} $\Cl(X,x)$
to be the factor group of the Weil divisor group
$\WDiv(X)$ by the subgroup $\PDiv(X,x)$ of all
Weil divisors being principal on a neighbourhood
of $x$. Proposition~\ref{prop:cartier-fwpp}
allows us to determine these groups for the toric
fixed points of a fake weighted projective space.

\begin{corollary}
\label{cor:local-class-groups}
Consider a fake weighted projective space
$Z = Z(P)$, the points $\mathbf{z}(i) \in Z$
and the divisor classes
$\omega_i = [D_i] \in K = \Cl(Z)$.
Then we have
$$
\ZZ \omega_i
\ = \
\{[D]; \ D \in \PDiv(Z,\mathbf{z}(i))\}
\ \subseteq \
K
\ = \
\Cl(Z).
$$
In particular, the local class groups $\Cl(Z,\mathbf{z}(i))$
of points $\mathbf{z}(i) \in Z$ and the respective group orders
are given as
$$
\Cl(Z,\mathbf{z}(i))
\ = \
K/\ZZ \omega_i,
\qquad
\vert \Cl(Z,\mathbf{z}(i)) \vert
\ = \
\vert \det(v_0,\ldots, \xcancel{v_i}, \ldots, v_n) \vert.
$$
Finally, the Picard group of $Z$, being the intersection over
all the local class groups in~$\Cl(Z)$, can be presented
as  
$$
\Pic(Z)
\ = \
\ZZ \omega_0 \cap \ldots \cap \ZZ \omega_n
\ \subseteq \
\Cl(Z).
$$
\end{corollary}

\begin{proof}
For computing the order of $\Cl(Z,\mathbf{z}(i))$,
we use~\cite[Lemma~2.1.4.1]{ArDeHaLa}.
All other assertions are direct consequences of
Proposition~\ref{prop:cartier-fwpp}.
\end{proof}

The \emph{local Gorenstein index}
of a point $x \in X$ is the smallest positive
integer $\iota(x)$ such that $\iota(x) \mathcal{K}_X$
is Cartier on a neighbourhood of $x \in X$.

\begin{remark}
For every $\QQ$-Gorenstein variety, the
Gorenstein index $\iota_X$ equals the least
common multiple of the local Gorenstein
indices $\iota(x)$, where $x \in X$.
\end{remark}

\begin{proposition}
\label{prop:fwps-gorind}
Consider a fake weighted projective space $Z = Z(P)$.
\begin{enumerate}
\item
The variety $Z$ is $\QQ$-factorial and thus $\QQ$-Gorenstein.
An anticanonical divisor of $Z$ is explicitly given by
$-\mathcal{K}_Z  = D_0 + \ldots + D_n$.
\item
For each $k = 0, \ldots, n$, there is a unique linear
form $u_k \in \QQ^n$ such that for all $j \ne k$ we have 
$\bangle{u_k,v_j} = 1$.
\item 
The local Gorenstein index of $\mathbf{z}(k) \in Z$ 
is the minimal positive integer $\iota_k$
such that $\iota_k u_k$
is a primitive vector in $\ZZ^n$.
\item
The Gorenstein index of $Z$ is the least common multiple
of the local Gorenstein indices
$\iota_0, \ldots, \iota_n$.
\end{enumerate}
\end{proposition}

\begin{proof}
Proposition~\ref{prop:fwpsqfact} ensures that
$Z$ is $\QQ$-factorial.
The statement on the anticanonical divisor is a special
case of~\cite[Thm.~8.2.3]{CoLiSc}.
This proves~(i).
Assertion~(ii) is clear because for each $k = 0, \ldots, n$
the family $(v_j; \ j \ne k)$ is a basis for $\QQ^n$.
We show~(iii). Let $\nu_k$ be a positive integer
such that the anticanonical divisor
$-\nu_k\mathcal{K}_Z$ from~(ii) is Cartier on the
open set $Z_i \subseteq Z$.
Then Proposition~\ref{prop:cartier-fwpp} tells us
that $-\nu_k\mathcal{K}_Z$ equals $\div(\chi^{\nu_ku_k})$
on $Z_k$ and $\nu_ku_k \in \ZZ^n$.
The claim follows. Assertion~(iv) is clear.
\end{proof}

\begin{definition}
Given an $n$-dimensional fake weighted projective
space $Z = Z(P)$, we call $u_0, \ldots, u_n$ from
Proposition~\ref{prop:fwps-gorind}~(iii) the
\emph{Gorenstein forms} of $Z$.
\end{definition}

The following elementary observation helps to identify
Gorenstein forms and local Gorenstein indices in the
case of fake weighted projective planes.

\begin{lemma}
\label{lem:locgordet}
Let $v_1 = (a,c)$ and $v_2 = (b,d)$ be non-collinear
primitive vectors in~$\ZZ^2$.
Moreover, consider
$$
u
\ := \
\left[
\frac{d-c}{ad-bc} , \, \frac{a-b}{ad-bc}
\right],
\qquad\qquad
\iota
\ := \ 
\frac{\vert ad-bc\vert}{\gcd(d-c,\,a-b)}.
$$
Then $\bangle{u,v_i} = 1$ holds for $i=1,2$
and $\iota$ is the unique positive integer
such that $\iota \cdot u$ is a primitive
vector in~$\ZZ^2$.
\end{lemma}

We provide the necessary statements for
our classification algorithm for fake weighted 
projective planes with given Gorenstein index.

\begin{lemma}
\label{lem:Gmatrix}
Fix a positive integer $\iota$ 
and consider the following matrix,
depending on three other positive
integers $a_0,a_1,a_2$:
$$
G
\ := \ 
G(\iota,a_0,a_1,a_2)
\ := \ 
\left[
\begin{array}{ccc}
\iota - a_0 & \iota & \iota
\\
\iota & \iota - a_1 & \iota
\\
\iota & \iota & \iota - a_2
\end{array}
\right].
$$
Then $G$ is of rank greater or equal to two and
it is of rank exactly two if and only if we
have the following identity of unit fractions:
$$
\frac{1}{\iota}
\ = \
\frac{1}{a_0} + \frac{1}{a_1} + \frac{1}{a_2} .
$$
If a triple of positive integers $(a_0,a_1,a_2)$
satisfies this identity, then the kernel of~$G$
is generated by the primitive vector
$$
w(a_0,a_1,a_2)
\ := \
\frac{1}{\gcd(a_1a_2,a_0a_2,a_0a_1)}
\left[
\begin{array}{c}
a_1a_2
\\
a_0a_2
\\
a_0a_1
\end{array}
\right]
\ \in \
\ZZ^3 .
$$
\end{lemma}

\begin{proof}
All statements made in the Lemma can be directly
verified via elementary computations.
\end{proof}

Now we are going to produce systematically
fake weighted projective planes $Z(P)$
of given Gorenstein index $\iota$ with
local Gorenstein indices
$\iota_0,\iota_1,\iota_2$.

\begin{proposition}
\label{prop:producefwpp}
Fix positive integers $\iota,\iota_0,\iota_1,\iota_2$ 
with $\iota = \lcm(\iota_0,\iota_1,\iota_2)$,
assume that $a_0,a_1,a_2 \in \ZZ_{\ge 1}$ satisfy
$$
\frac{1}{\iota}
\ = \
\frac{1}{a_0} + \frac{1}{a_1} + \frac{1}{a_2}
$$
and let $w:=w(a_0,a_1,a_2)=(w_0,w_1,w_2)$ be the primitive
generator of the kernel of the matrix~$G(\iota,a_0,a_1,a_2)$.
Consider the $2 \times 3$ matrices of the form
$$
P
\ := \
\left[
\begin{array}{ccc}
1 & ax +1 & - \frac{w_0+w_1+axw_1}{w_2}
\\
0 & x\iota_2 & - \frac{x \iota_2 w_1}{w_2}
\end{array}  
\right],
\qquad
\begin{array}{l}
\scriptstyle
x \in \ZZ_{\ge 1}, \ xw_1 \mid (w_0+w_1+w_2)\iota_1,
\\
\scriptstyle
a \in \ZZ, \ -\frac{1}{x}   \le  a < \iota_2 - \frac{1}{x}.
\end{array}  
$$
If the columns of $P$ are primitive vectors in $\ZZ^2$,
then it defines a fake weighted projective plane
$Z = Z(P)$.
Moreover, if each of
$$
\iota_0 \cdot
\left[
\begin{array}{c}
-\frac{w_1+w_2}{w_0} 
\\
\frac{w_0+w_1+w_2+ax(w_1+w_2)}{x\iota_2w_0}    
\end{array}  
\right],
\quad
\iota_1 \cdot
\left[
\begin{array}{c}
1
\\
- \frac{w_0+w_1+w_2+axw_1}{x\iota_2w_1}
\end{array}  
\right],  
\quad
\iota_2 \cdot
\left[
\begin{array}{c}
1
\\
-\frac{a}{\iota_2}
\end{array}  
\right]
$$
is a primitive vector in $\ZZ^2$, then $Z$
is of Gorenstein index $\iota$ with the local
Gorenstein indices $\iota_k$ of the toric
fixed points $\mathbf{z}(k)$, where $k = 0,1,2$.
\end{proposition}

\begin{proof}
We show that the columns of $P$ generate $\QQ^2$
as a cone.
Given $v \in \QQ^2$, we can clearly write it
as a linear combination over the columns of $P$,
that means $v = P \cdot u$ for some $u \in \QQ^3$.
Now observe $P \cdot w = 0$.
Thus, since all entries of $w$ are strictly
positive, we obtain a positive combination
$v = P \cdot (u+\xi w)$ by choosing $\xi \in \QQ_{0}$
big enough. 
Thus, if the columns of $P$ are primitive
integral vectors in~$\ZZ^2$, then it
defines a fake weighted projective plane.
If also the three displayed vectors are
integral and primitive, then, using
Lemma~\ref{lem:locgordet}, one directly
checks that these are the $\iota_k$-fold
local Gorenstein form $u_k$ of the toric
fixed points $\mathbf{z}(k)$.
\end{proof}

Finally, we show that the above example yields
indeed all fake weighted projective planes
of given Gorenstein index.

\begin{proposition}
\label{prop:allfwpp}
Consider a fake weighted projective plane~$Z$
of Gorenstein index~$\iota$.
Then $Z \cong Z(P)$ with a matrix $P$ as
provided by Proposition~\ref{prop:producefwpp}.
\end{proposition}  

\begin{proof}
Proposition~\ref{prop:allfwpsviaP} allows us
to assume $Z = Z(P)$ with a $2 \times 3$
matrix $P$.
We claim that the divisor class group of $Z$ is
of the form
$$
\Cl(Z)
\ \cong \
\ZZ^3 / P^*\ZZ^2
\ \cong \ 
\ZZ \times \ZZ/n\ZZ.
$$
Indeed, by Remark~\ref{rem:well-formed}, any two
of the weights $\omega_0,\omega_1,\omega_2$
generate $\ZZ \times \Gamma$ as a group,
forcing~$\Gamma$ to be cyclic.
Consequently, the degree matrix
is given as
$$
Q
\ = \
\left[
\begin{array}{ccc}
w_0 & w_1 & w_2
\\
\eta_0 & \eta_1 & \eta_2
\end{array}
\right],
\qquad
w_0,w_1,w_2 \in \ZZ_{\ge 1},
\quad                   
\eta_0,\eta_1,\eta_2 \in \ZZ/n\ZZ.
$$ 
As $Z$ is of Gorenstein index $\iota$, the divisor
$-\iota\mathcal{K}_Z$ is Cartier and thus
Proposition~\ref{prop:cartier-fwpp}
provides us with positive integers $a_0,a_1,a_2$
such that 
$$
\left[
\begin{array}{ccc}
\iota - a_0 & \iota & \iota
\\
\iota & \iota - a_1 & \iota
\\
\iota & \iota & \iota - a_2
\end{array}
\right]
\cdot
\left[
\begin{array}{c}
w_0
\\
w_1
\\
w_2
\end{array}     
\right]
\ = \ 
0.
$$
Thus, Remark~\ref{rem:well-formed} and
Lemma~\ref{lem:Gmatrix} yield
$(w_0,w_1,w_2) = w(a_0,a_1,a_2)$.
Now, let us see why $P$ can be chosen as in
Proposition~\ref{prop:producefwpp}.
First, we may assume
$$
P
\ = \ 
\left[
\begin{array}{rrr}
1   &   \alpha    &    \beta
\\
0   &   \gamma    &    \delta
\end{array}
\right] ,
\qquad
\alpha, \beta, \gamma, \delta \in \ZZ,
\quad
0 \le \alpha < \gamma.
$$
By Lemma~\ref{lem:locgordet}, the local Gorenstein
index $\iota_2$ divides $\gamma = \det(v_0,v_1)$.
Thus, $\gamma=\iota_2x$ with an integer $x \ge 1$.
Moreover, the Gorenstein form $u_2 \in \QQ^2$ is
given as
$$
u_2
\ = \
[1,q]
\ = \
\left[ 1, \, \frac{1-\alpha}{\gamma} \right].
$$
Thus, $\iota_2 q$ is an integer.
Using $\gamma = \iota_2 x$, we see $\alpha = 1 - \iota_2q x$.
Thus, with $a := -\iota_2q$, the first two columns
are as wanted. Resolving $P \cdot w = 0$ for $\beta$
and $\gamma$, we arrive at
$$
P
\ = \
\left[
\begin{array}{ccc}
1 & 1 + ax & - \frac{w_0+(1+ax)w_1}{w_2}
\\
0 & \iota_2 x & - \frac{\iota_2 x w_1}{w_2}
\end{array}  
\right].
$$
The condition $xw_1 \mid \iota_1 (w_0+w_1+w_2)$ holds because
the $\iota_1$-fold of the Gorenstein form
$u_1$ is integral and $-1/x \le a < \iota_2 -1/x$ reflects
$0 \le \alpha < \gamma$.
\end{proof}

We come to the algorithmic part. The key to efficient
classification algorithms is the computability of
all presentations of any given rational number
$0 < q < 1$ as a sum of a fixed number $n+1$
of unit fractions:
$$
q \ = \ [a_0,\ldots,a_n] \ := \ \frac{1}{a_0} + \ldots  + \frac{1}{a_n},
\qquad
a_0, \ldots, a_n \in \ZZ_{\ge 1}.
$$
We refer to $[a_0,\ldots,a_n]$ as a \emph{UFP of length $n+1$}
of $q$.
The list of all UFP of length 2 of a given $q$
is quickly determined:
for each $1/q < a_0 \le 2/q$ check 
$a_1 := a_0/(a_0q-1) \in \ZZ$ and
if so, add $[a_0,a_1]$ to the list.

Our first algorithm computes in particular all weight
vectors $(w_0,w_1,w_2)$ of fake weighted projective
planes $Z(P)$ of given Gorenstein index $\iota$
arising from Construction~\ref{constr:fwps2}.
Recall that the vector $(w_0,w_1,w_2)$ lists the
weights of the covering weighted projective plane
and that it is \emph{well-formed} in the sense that
any two of its entries are coprime.

\begin{algorithm}[Computation of well-formed weight vectors]
\label{algo:wfwv}
\emph{Input:}
A positive integer $\iota$, the prospective
Gorenstein index.
\emph{Algorithm:}
\begin{itemize}
\item[$\bullet$]
open an empty list $W$ for the weight vectors;
\item[$\bullet$]
for $a_0$ from $\iota+1$ to $3 \iota$ do
\begin{itemize}
\item[$\bullet$]
determine $q := \frac{1}{\iota} -  \frac{1}{a_0}$;
\item[$\bullet$]
compute the list $Q$ of all UFP $[a_1,a_2]$ of $q$;
\item[$\bullet$] 
for every UFP $[a_1,a_2]$ from the list $Q$ do 
\begin{itemize}
\item[$\bullet$]
determine 
$w := \frac{1}{\gcd(a_1a_2,a_0a_2,a_0a_1)} (a_1a_2, \, a_0a_2, \, a_0a_1)$;
\item[$\bullet$]
if the entries of $w$ are pairwise coprime,
then add $w$ to $W$;
\end{itemize}
\item[$\bullet$] end do;
\end{itemize}
\item[$\bullet$] end do;
\end{itemize}
\emph{Output:} The list $W$. Then $W$ contains in particular the
weight vectors $(w_0,w_1,w_2)$ of all fake weighted projective
planes $Z(P)$ of Gorenstein index $\iota$
arising from Construction~\ref{constr:fwps2}.
\end{algorithm}

\begin{proof}
Obviously, the algorithm terminates.
The fact that the output list $W$
contains the weight vectors of all
fake weighted projective
planes $Z(P)$ of Gorenstein index $\iota$
arising from Construction~\ref{constr:fwps2}
follows from Propositions~\ref{prop:producefwpp}
and~\ref{prop:allfwpp}.     
\end{proof}

We come to the classification algorithm.
For two defining matrices $P$ and $P'$ as in
Construction~\ref{constr:fwps2}, we
write $P \sim P'$ if
$P' = A \cdot P \cdot S$ holds with a unimodular
matrix $A$ and a permutation matrix $P$.
Recall from Proposition~\ref{prop:fwpsPequiv}
that $P \sim P'$ holds if and only if the varieties
$Z(P)$ and $Z(P')$ are isomorphic to each other.

\begin{algorithm}[Classification of fake weighted projective planes]
\label{alg:classfwpp}  
\emph{Input:}
A positive integer $\iota$, the prospective Gorenstein index.
\emph{Algorithm:}
\begin{itemize}
\item
use Algorithm~\ref{algo:wfwv} to compute the list $W$
of weight vectors for~$\iota$;
\item
open an empty list $L$ for the defining matrices $P$; 
\item
for each triple $\iota_0 \le \iota_1 \le \iota_2$ of positive
integers with $\lcm(\iota_0,\iota_1,\iota_2)=\iota$ do 
\begin{itemize}
\item[$\bullet$]
for every weight vector $(w_0,w_1,w_2)$ from the list $W$ do
\begin{itemize}
\item[$\bullet$]
for every pair of integers $(x,a) \in \ZZ_{\ge 1} \times \ZZ_{\ge - 1}$
satisfying the conditions
$$
\hspace*{2cm}
\scriptstyle
xw_1 \mid \iota_1(w_0+w_1+w_2), \quad
- \frac{1}{x} \le a \le \iota_{2} - \frac{1}{x}, \quad
\gcd(a,\iota_2) = 1 ,
$$
plug the actual values of the variables
$\iota_0,\iota_1,\iota_2,w_0,w_1,w_2,x$ and~$a$ into
$$
\hspace*{2cm}
\begin{array}{lcl}
P
& := &
\left[
\begin{array}{ccc}
\scriptstyle
  1
  &
\scriptstyle
    ax +1
  &
\scriptstyle
    - \frac{w_0+w_1+axw_1}{w_2}
\\
\scriptstyle
  0
  &
\scriptstyle
    x\iota_2
  &
\scriptstyle
    - \frac{x \iota_2 w_1}{w_2}
\end{array}  
\right],
\\[10pt]
u_0
& := &
\iota_0 \cdot
\left[
\begin{array}{c}
\scriptstyle
-\frac{w_1+w_2}{w_0} 
\\
\scriptstyle
\frac{w_0+w_1+w_2+ax(w_1+w_2)}{x\iota_2w_0}    
\end{array}  
\right],  
\\[10pt]
u_1
& := &
\iota_1 \cdot
\left[
\begin{array}{c}
\scriptstyle
1
\\
\scriptstyle
- \frac{w_0+w_1+w_2+axw_1}{x\iota_2w_1}
\end{array}  
\right];
\end{array}  
$$
\item[$\bullet$]
if the columns of $P$ and $u_0$, $u_1$ are primitive integral
and $P \not\sim P'$ for all $P'$ of $L$, then add~$P$
to $L$;
\end{itemize}  
\item[$\bullet$]
end do;
\end{itemize}
\item[$\bullet$]
end do;
\end{itemize}
\emph{Output:} The list $L$. Then, for each $P$ from $L$, the
associated fake weighted projective plane $Z(P)$ is Gorenstein
index $\iota$.
Moreover any fake weighted projective plane of Gorenstein
index $\iota$ is isomorphic to $Z(P)$ for precisely one $P$ from $L$.
\end{algorithm}

\begin{proof}
Obviously, the algorithm terminates.
The statements on $L$ hold due to
Propositions~\ref{prop:fwpsPequiv},
\ref{prop:producefwpp} and~\ref{prop:allfwpp}.     
\end{proof}

\section{Constructing $\KK^*$-surfaces of Picard number one}
\label{sec:basics-kstar}

In Construction~\ref{constr:kstarsurf} we
produce rational projective $\KK^*$-surfaces $X$
of Picard number one coming embedded into
a fake weighted projective space $Z$,
where they are given as the closure of a linear
subvariety of the acting torus.
Moreover, the $\KK^*$-action of $X \subseteq Z$ is given
by a one-dimensional subtorus of the acting torus
$\TT^n \subseteq Z$. We begin with a concrete example.

\begin{example}
\label{ex:e6-cubic-1}
The weighted projective space $Z = \PP(1,3,2,3)$
is obtained via Construction~\ref{constr:fwps2}
as the variety $Z = Z(P)$ given by the defining
matrix 
$$
P
\ = \
\left[
\begin{array}{rrrr}
-3 & -1 & 3 & 0
\\
-3 & -1 & 0 & 2
\\
-2 & -1 & 1 & 1
\end{array}
\right].
$$
Consider the homomorphism $p \colon \TT^4 \to \TT^3$
of tori associated with $P$ and the quotient map
$\pi \colon \KK^4 \setminus \{0\} \to Z$ from~\ref{constr:fwps2}.
Remark~\ref{rem:fwpp2tv2} provides us with the identification 
$$
\TT^3 \ = \ p(\TT^4) \ = \ \pi(\TT^4) \ \subseteq \ Z,
$$
so that we can regard $\TT^3$ as the acting torus of $Z$.
In terms of the coordinates $s_1,s_2,s_3$ of $\TT^3$, set 
$h := 1+s_1+s_2$, look at its zero set
$V(h) \subseteq \TT^3$ and the closure
$$
X
\ := \
X(P)
\ := \
\overline{V(h)}
\ \subseteq \
Z
\ = \
Z(P).
$$
Then $X$ is an irreducible, rational surface in $Z$.
We install a $\KK^*$-action on~$X$.
First, take the action of $\KK^*$ on $\TT^3$
given in terms of the coordinates $s_1,s_2,s_3$
as 
$$
t \cdot s \ := \ (s_1,s_2,ts_3).
$$
Since $h$ doesn't depend on $s_3$, its zero
set $V(h) \subseteq \TT^3$ stays invariant.
Now, $\TT^3 \subseteq Z$ is the acting torus
and thus the $\KK^*$-action extends to $Z$,
leaving $X$ invariant.
\end{example}

In fact, this introductory example turns out to be
the well known $E_6$-singular cubic surface;
see~\cite[Examples~3.4.3.3, 3.4.4.2, 3.4.4.10 and 5.4.3.5]{ArDeHaLa}.
The following construction distinguishes
the types~(ee) and~(ep); as we will see,
the first one admits only isolated $\KK^*$-fixed
points, whereas the second one always comes with
a curve of $\KK^*$-fixed points.

\begin{construction}[Rational projective $\KK^*$-surfaces
of Picard number one]
\label{constr:kstarsurf}
Consider defining $(r+1) \times (r+2)$ matrices
$P$ as in Construction~\ref{constr:fwps2} 
of the shape

\medskip
\noindent
\emph{Type (ee)}
$$
P
\ = \
[v_{01},v_{02},v_1,\ldots,v_r]
\  = \
\left[
\begin{array}{ccccc}
-l_{01} & -l_{02} & l_1 &        & 
\\
\vdots & \vdots &     & \ddots
\\
-l_{01} & -l_{02} &     &        &  l_r 
\\
d_{01} & d_{02} & d_1 & \ldots &  d_r
\end{array}                                     
\right],
\quad
\begin{array}{c}
\scriptstyle
l_{01}, l_{02} \ge 1,
\\[1ex]
\scriptstyle
l_1,\ldots, l_r \ge 2,
\\[1ex]
\scriptstyle
\frac{d_{01}}{l_{01}} > \frac{d_{02}}{l_{02}},
\end{array}
$$

\medskip
\noindent
\emph{Type (ep)}
$$
P
\ = \
[v_0,v_1,\ldots,v_r,v^-]
\ = \
\left[
\begin{array}{ccccc}
-l_0 & l_1 &        &       & 0
\\
\vdots &     & \ddots  &      & \vdots
\\
-l_0 &     &        &  l_r  & 0
\\
d_0 & d_1 & \ldots &  d_r & -1
\end{array}                                     
\right],
\quad 
\begin{array}{c}
\scriptstyle
l_0,l_1,\ldots, l_r \ge 2.
\end{array}
$$

\medskip
\noindent
Let $Z(P)$ be the associated fake weighted projective
space and $\TT^{r+1} \subseteq Z(P)$ the acting torus.
Fix pairwise different
$1=\lambda_2,\lambda_3, \ldots, \lambda_r \in \KK^*$
and set
$$
h_i(s_1,\ldots,s_{r+1})
\ := \
\lambda_i + s_1 + s_i,
\quad
i = 2, \ldots, r,
$$
where $s_1,\ldots,s_{r+1}$ are the coordinate functions
on $\TT^{r+1}$.
Passing to the closure of the set of common zeroes
of $h_2,\ldots,h_r$ in $\TT^{r+1}$, we obtain a projective
surface
$$
X(P)
\ := \ 
\overline{V(h_2, \ldots, h_r)}
\ \subseteq \ 
Z(P).
$$
As the functions $h_2,\ldots,h_r$ don't depend on the
last coordinate $s_{r+1}$ of $\TT^{r+1}$, we see that
$\KK^*$ acts effectively as a subtorus of $\TT^{r+1}$
on $X(P) \subseteq Z(P)$ via
$$
t \cdot x \ = \ (1,\ldots, 1,t) \cdot x.
$$
We refer to $P$ as the \emph{defining matrix} of
$X(P)$ and, if there is need to specify, we call
$(P,\lambda)$ with $\lambda = (\lambda_3,\ldots,\lambda_r)$
the \emph{defining data} of $X(P,\lambda) := X(P)$.
\end{construction}

We also need a description of  $X = X(P)$ 
in terms of homogeneous coordinates of its
ambient fake weighted projective space $Z = Z(P)$.
For this, we suitably \emph{homogenize}
the defining equations $h_2, \ldots, h_r$
for $X \cap \TT^{r+1}$, which then provides us
with defining equations $g_2, \ldots, g_r$
in homogeneous coordinates.
Let us first take another look at the
introductory example.

\begin{example}
\label{ex:e6-cubic-2}
Consider $X=X(P)$ in $Z=Z(P)$ from
Example~\ref{ex:e6-cubic-1}.
We relate the homogeneous coordinates on $Z$
to the coordinates of $\TT^3 \subseteq Z$
via the homomorphism of tori
$p \colon \TT^4 \to \TT^3$ associated with~$P$,
concretely given as
$$
p \colon \TT^4 \ \to \ \TT^3,
\qquad
(t_{01},t_{02},t_1,t_2)
\ \mapsto \
\left(
\frac{t_1^3}{t_{01}^3t_{02}}, \,
\frac{t_2^2}{t_{01}^3t_{02}}, \,
\frac{t_1t_2}{t_{01}^2t_{02}}
\right),
$$
where the particular numbering of coordinates on $\TT^4$
indicates that $t_{01}$ and $t_{02}$ form the denominator
monomial of the first two components.
Pulling back the defining equation  $h = 1+s_1+s_2$ 
gives the $H$-invariant Laurent polynomial
$$
p^*h(t_{01},t_{02},t_1,t_2)
\ = \
1 + \frac{t_1^3}{t_{01}^3t_{02}} + \frac{t_2^2}{t_{01}^3t_{02}}.
$$
By definition, the hypersurface 
$V(p^*h) \subseteq \TT^4$
maps onto $V(h) \subseteq \TT^3$ via $p$.
We are interested in the closure $\bar X \subseteq \KK^4$
of $V(p^*h) \subseteq \TT^4$.
This is the hypersurface $\bar X = V(g)$ defined by
the $H$-homogeneous polynomial
$$
g(t_{01},t_{02},t_1,t_2)
\ = \
t_{01}^3t_{02}p^*h
\ = \
t_{01}^3t_{02} + t_1^3 + t_2^2.
$$
\end{example}

\begin{construction}
\label{constr:homogenization}
Consider $X = X(P,\lambda)$ in $Z = Z(P)$
and $K = \ZZ^{r+2}/P^*\ZZ^{r+1}$.
According to the type of $P$,
the \emph{$K$-homogeneous coordinate ring}
of $Z(P)$ is
$$
\begin{array}{llll}
\text{(ee)} \quad
&
\mathcal{R}(Z) := \KK[T_{01}, T_{02}, T_1, \ldots, T_r],
&
\deg(T_{0j}) = Q(e_{0j}),
&
\deg(T_{i})  = Q(e_{i}),
\\[1ex]
\text{(ep)} \quad
&
\mathcal{R}(Z) := \KK[T_0, T_1, \ldots, T_r, T^-],
&
\deg(T_{i}) = Q(e_{i}),
&
\deg(T^-)  = Q(e^-),
\end{array}
$$
where $Q \colon \ZZ^{r+2} \to K$ denotes
the projection.
According to the type of $P$,
we define for $i = 2,\ldots,r$
the \emph{$K$-homogenization}
$g_i = g_{i,\lambda}$ of the affine form
$h_i$ as
$$
\begin{array}{ll}
\text{(ee)} \quad
&
g_i
\ := \
T_{01}^{l_{01}} T_{02}^{l_{02}} p^*h_i
\ = \
\lambda_i T_{01}^{l_{01}} T_{02}^{l_{02}} + T_1^{l_1} + T_i^{l_i}
\ \in \
\KK[T_{01}, T_{02}, T_1, \ldots, T_r],
\\[1ex]
\text{(ep)} \quad
&
g_i
\ := \
T_0^{l_0} p^*h_i
\ = \
\lambda_i T_0^{l_0} + T_1^{l_1} + T_i^{l_i}
\ \in \
\KK[T_0, T_1, \ldots, T_r, T^-],
\end{array}
$$                    
where $p \colon \TT^{r+2} \to \TT^{r+1}$
denotes the homomorphism of tori given
by the defining matrix $P$.
Now, consider the set of common zeroes
$$
\bar X
\ := \
V(g_2, \ldots, g_r)
\ \subseteq \
\KK^{r+2},
$$
set $\hat X := \bar X \setminus \{0\}$
and let $\pi \colon \KK^{r+2} \setminus \{0\} \to Z$
be the quotient map for the $H$-action
on $\KK^{r+2} \setminus \{0\}$.
Then we can describe $X \subseteq Z$
in homogeneous coordinates as
$$
X
\ = \
V(g_2, \ldots, g_r)
\ := \
\pi(\hat X)
\ \subseteq \
Z.
$$
Note that restricting $\pi \colon \KK^{r+2} \setminus \{0\} \to Z$
yields a quotient map $\pi \colon \hat X \setminus \{0\} \to X$
for the action of the quasitorus $H$ on $\hat X$.
\end{construction}

\begin{proposition}
\label{prop:X(P)normal}
Any $X=X(P)$ arising from 
Construction~\ref{constr:kstarsurf}
is an irreducible, rational, normal
projective surface.
\end{proposition}

\begin{proof}[Proof of Construction~\ref{constr:homogenization}
and Proposition~\ref{prop:X(P)normal}]
Obviously, $V(h_2, \ldots, h_r)$
is an irreducible, rational surface
in $\TT^{r+1}$. 
Thus, passing to the closure in the
fake weighted projective space $Z$
yields an irreducible, rational,
projective surface $X$.
In order to see that $X$ is normal,
we consider the $H$-invariant
closed subset
$$
\bar X
\ = \
V(g_2, \ldots, g_r)
\ \subseteq \
\KK^{r+2}.
$$
The Jacobian of $J(g_2, \ldots, g_r)$
is of full rank everywhere in $\bar X$
outside a closed subset of codimension
at least two.
Thus, $\bar X$ is normal.
Moreover, every $H$-orbit closure contains
$0 \in \KK^{r+2}$ and thus all components
of $\bar X$ intersect in the origin.
By normality, $\bar X$ is irreducible.
Consequently,
$$
X
\ = \
\overline{V(h_2,\ldots,h_r)}
\ = \
\overline{p(V(g_2,\ldots,g_r) \cap \TT^{r+2})}
\ = \
\pi(\bar X \setminus \{0\})
$$
holds inside $Z$.
In particular, the surface $X$ is the quotient
of the action of $H$ on the normal variety
$\bar X \setminus \{0\}$.
Using the respective universal properties of the
normalization and the quotient, we see that
the quotient surface $X$ must be normal as well.
\end{proof}

\begin{remark}
\label{rem:families}
Construction~\ref{constr:kstarsurf}
delivers in fact \emph{$(r-2)$-dimensional families}
of $\KK^*$-surfaces. Given a defining matrix $P$,
the parameter space will be
$$
U(P)
\ := \ 
\{\lambda = (\lambda_3,\ldots,\lambda_r) \in \TT^{r-2}; \ \lambda_i \ne 1, \ 3 \le i \le r, \ \lambda_i \ne \lambda_j, \ 3 \le i < j \le r\}.
$$
We obtain an action of $H$ on $\KK^{r+2} \times U(P)$
by trivially extending the one from~\ref{constr:kstarsurf}
to the second factor.
With $g_i = g_{i,\lambda}$,
where $i=2,\ldots,r$ and $\lambda \in U(P)$, set
$$
\tilde {\mathcal X}(P)
\ := \
\{ (z,\lambda) \in \KK^{r+2} \times U(P); \ g_{2,\lambda}(z) = \ldots = g_{r,\lambda}(z) = 0\}
\ \subseteq \
\KK^{r+2} \times U(P).
$$
Then $\tilde {\mathcal X}(P)$ and also 
$\tilde {\mathcal X}(P) \cap (\{0\} \times U(P))$
are invariant under the
action of $H$ on $\KK^{r+2} \times U(P)$.
Passing to the quotients yields a commutative diagram
$$
\xymatrix{
{\tilde {\mathcal X}(P) \cap (\{0\} \times U(P))}
\ar[rr]
\ar[d]_{/H}
&&
{(\KK^{r+2} \setminus \{0\}) \times U(P)}
\ar[d]^{(z,\lambda) \mapsto (H \cdot z, \lambda)}_{/H}
\\
{\mathcal X} (P)
\ar[rr]
\ar[dr]
&&
Z(P) \times U (P)
\ar[dl]^{(x,\lambda) \mapsto \lambda}
\\
& U(P) &
}
$$
where the horizontal arrows are closed embeddings.
The variety ${\mathcal X}(P)$ is normal, of dimension $r$
and for each $\lambda \in U(P)$ the fiber ${\mathcal X}(P)_\lambda$
is isomorphic to $X(P,\lambda)$.
\end{remark}

Construction~\ref{constr:kstarsurf} fits
into~\cite[Constr.~5.4.1.3 and Thm.~5.4.1.5]{ArDeHaLa}
treating also $\KK^*$-surfaces of higher Picard numbers.
The concrete link is given as follows.

\begin{remark}
\label{rem:XAP}
Let $X(P,\lambda)$ be a $\KK^*$-surface provided by 
Construction~\ref{constr:kstarsurf} and consider
the $2 \times (r+1)$ matrix 
$$
A
\ = \
[a_0,\ldots,a_r]
\ = \ 
\left[
\begin{array}{cccccc}
1 & 0 & -1 & -\lambda_3 & \ldots & - \lambda_r   
\\
0 & 1 & -1 & -1 & \ldots & -1   
\end{array}
\right].
$$
Then the $g_2, \ldots, g_r$ from
Construction~\ref{constr:homogenization}
form a basis of the vector subspace
of the $K$-homogeneous coordinate
ring of $Z(P)$ generated by the trinomials
$$
g_{i_1,i_2,i_3}
\ = \ 
\det
\left[
\begin{array}{ccc}                    
T_{i_1}^{l_{i_1}} & T_{i_2}^{l_{i_2}} & T_{i_3}^{l_{i_3}}                    
\\
a_{i_1} & a_{i_2} & a_{i_3}
\end{array}
\right],
\quad
0 \le i_1 < i_2 < i_3 \le r,
$$
where $T_0^{l_0} := T_{01}^{l_{01}} T_{02}^{l_{02}}$
in the case (ee).
In particular, $X(P,\lambda)$ equals the
$\KK^*$-surfaces $X(A,P)$ given
by~\cite[Constr.~5.4.1.3, Thm.~5.4.1.5]{ArDeHaLa}
and~\cite[Constr.~5.2]{Hae}.
\end{remark}

Further generalizations, as presented
in~\cites{ArDeHaLa,HaHiWr},
extend to higher dimension and complexity.
The main point for connecting to those is to
add a third data, a collection~$\Phi$ of cones,
which is essential from dimension three on.

\begin{remark}
\label{rem:XRFPhi}
Let $X = X(P,\lambda) = X(A,P)$ arise from
Construction~\ref{constr:kstarsurf} with~$A$
as in Remark~\ref{rem:XAP}. Consider $K = \ZZ^{r+2} / P^* \ZZ^{r+1}$
and set
$$
\Phi := \{ \tau \},
\quad
\tau := \QQ_{\ge 0} \subseteq \QQ = \QQ \otimes_\ZZ K.
$$
Then $X(P,\lambda)$ is the variety $X(A,P,\Phi)$
produced by~\cite[Constr.~3.4.3.6]{ArDeHaLa}.
Moreover, the graded ring $R(A,P)$ from there
equals the $K$-graded factor ring
$$
R := \mathcal{R}(Z)/\bangle{g_2,\ldots,g_r}.
$$ 
Finally, with the system of generators $\mathfrak{F}$
of $R$ given by the classes of the variables
$T_{ij},T_i,T^-$, our $X(P,\lambda)$ equals
$X(R,\mathfrak{F},\Phi)$
from~\cite[Constr.~3.2.1.3]{ArDeHaLa}.
\end{remark}

We gather the necessary basic geometric properties
of the $\KK^*$-surfaces arising from
Construction~\ref{constr:kstarsurf}.
Whenever we refer to~\cite{ArDeHaLa}, the above
Remarks~\ref{rem:XAP} and~\ref{rem:XRFPhi}
yield the explicit connection.

\begin{proposition}
\label{prop-basic-geom-ee}
Consider $X=X(P)$ in $Z=Z(P)$ of type (ee) as provided
by Construction~\ref{constr:kstarsurf}
and the coordinate divisors on $Z$, given
as
$$
D_Z^{0j} = V(T_{0j}) \subseteq Z,
\quad
j = 1,2,
\qquad
D_Z^i = V(T_i)  \subseteq  Z,
\quad
i = 1, \ldots, r.
$$
Cutting down the coordinate divisors to the
subvariety $X \subseteq Z$ yields
$\KK^*$-invariant prime divisors on $X$,
namely
$$
D_X^{0j} := X \cap D_Z^{0j} \subseteq X,
\quad
j = 1,2,
\qquad
D_X^i := X \cap D_Z^i \subseteq X,
\quad
i = 1, \ldots, r.
$$
Each of the  $D_X^{0j}$ and $D_X^i$ is
the closure of a non-trivial $\KK^*$-orbit
having $C(l_{0j})$ and~$C(l_i)$ as its
isotropy group, respectively.
The $D_X^{0j}$ and $D_X^i$ intersect as follows
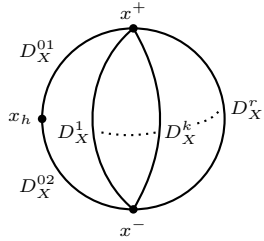
\begin{center}
\begin{tikzpicture}[scale=0.6]
\sffamily
\coordinate(oo) at (0,0);
\coordinate(source) at (0,2);
\coordinate(sink) at (0,-2);
\coordinate(x01) at (-2,0);
\coordinate(x0aa) at (-2.1, 1.5);
\coordinate(x0zz) at (-2.1,-1.5);
\coordinate(xraa) at (-1.3,-.3);
\coordinate(xrab) at (1.05,-.3);
\coordinate(xrzz) at (2.5,.2);
\coordinate(x0r) at (2,0);
\coordinate(dd1) at (-.7,-.3);
\coordinate(dd2) at (0.5,-.3);
\coordinate(dd3) at (1.4,-.1);
\coordinate(dd4) at (1.9,.2);

\path[fill, color=black] (source) circle (.6ex) 
node[above, font=\scriptsize]{$x^+$};
\path[fill, color=black] (sink) circle (.6ex) 
node[below, font=\scriptsize]{$x^-$};
\path[fill, color=black] (x01) circle (.6ex)
node[left, font=\scriptsize]{$x_h$};
\draw[thick, color=black] (oo) circle (2);
\node[align=center, font=\scriptsize] (D01) at (x0aa)  {$D_X^{01}$};
\node[align=center, font=\scriptsize] (D0n0) at (x0zz)  {$D_X^{02}$};
\draw[thick, bend right=50] (source) to (sink) node[]{};
\node[align=center, font=\scriptsize] (Dr1) at (xraa)  {$D_X^1$};
\draw[thick, bend left=30] (source) to (sink) node[]{};
\node[align=center, font=\scriptsize] (Dr1) at (xrab)  {$D_X^k$};
\node[align=center, font=\scriptsize] (Drnr) at (xrzz)  {$D_X^r$};
\draw[thick, dotted, bend right=10] (dd1) to (dd2) node[]{};
\draw[thick, dotted, bend right=18] (dd3) to (dd4) node[]{};
\end{tikzpicture}
\end{center}
The points $x^+ = \mathbf{z}(02)$ and
$x^- = \mathbf{z}(01)$ together with
the unique intersection point
$x_h \in D_X^{01} \cap D_X^{02}$
form the fixed point set of the $\KK^*$-surface $X$.
\end{proposition}

\begin{proof}
We indicate why the $D_X^{0j}$ and $D_X^k$ are indeed
prime divisors.
Consider $\bar X = V(g_2, \ldots, g_r)$ in $\KK^{r+2}$.
Then we obtain $H$-invariant hypersurfaces $\bar X \cap V(T_{0j})$
and $\bar X \cap V(T_k)$ in $\bar X$.
It turns out that $H$ transitively permutes
the components of each of these hypersurfaces
and thus, their images
$D_X^{0j}$ and $D_X^k$ are prime divisors;
see~\cite[Props.~1.5.3.3, 3.2.2.5, Thm.~3.4.3.4]{ArDeHaLa},
compare also~\cite[Prop.~10.9]{HaHe}.

We determine the isotropy groups of the non-trivial
$\KK^*$-orbits of the $D_X^{0j}$ and~$D_X^k$.
Consider exemplarily a point  $x \in D_X^{01}$
from the non-trivial $\KK^*$-orbit.
Then $x = [z]$, where $z_{01}=0$ and $z_{02},z_1,\ldots,z_r$
are all non-zero.
Then we have
$$
\KK^*_{x}
\ = \
\TT^{r+1}_x \cap \KK^*
\ = \
\{(s^{-l_{01}},\ldots,s^{-l_{01}},s^{d_{01}}); \ s \in \KK^*\}
\ \cap \
\{(1,\ldots,1,t); \ t \in \KK^*\},
$$
where we allow ourselves to write just $\KK^*$ for
the subgroup $\{1\} \times \ldots \times \{1\} \times \KK^*$
of the acting torus $\TT^{r+1}$ of $Z(P)$ and
Proposition~\ref{prop:fwpstoric} provides us with
the description of the isotropy group $\TT^{r+1}_x$.
\end{proof}

\begin{proposition}
\label{prop-basic-geom-ep}
Consider $X=X(P)$ in $Z=Z(P)$ of type (ep) as provided
by Construction~\ref{constr:kstarsurf}
and the coordinate divisors on $Z$, given
as
$$
D_Z^i = V(T_i)  \subseteq  Z,
\quad
i = 0, \ldots, r,
\qquad
D_Z^- = V(T^-)  \subseteq  Z.
$$
Cutting down the coordinate divisors to the
subvariety $X \subseteq Z$ yields
$\KK^*$-invariant prime divisors on $X$,
namely
$$
D_X^i := X \cap D_Z^i \subseteq X,
\quad
i = 0, \ldots, r,
\qquad
D_X^- :=  X \cap D_Z^- \subseteq X.
$$
Each of the  $D_X^i$ is the closure of a
non-trivial $\KK^*$-orbit having $C(l_i)$
as its isotropy group and $D_X^-$ consists of
fixed points.
The $D_X^{i}$ and $D_X^-$ intersect as follows
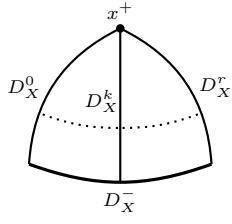
\begin{center}
\begin{tikzpicture}[scale=0.6]
\sffamily
\coordinate(source) at (0,1);
\coordinate(sink) at (0,-2);
\coordinate(sinkl) at (-2,-2);
\coordinate(sinkm) at (0,-2.4);
\coordinate(sinkr) at (2,-2);
\coordinate(d0) at (-2.1,-0.3);
\coordinate(dk) at (-0.4,-0.6);
\coordinate(dr) at (2.1,-0.3);
\coordinate(dminus) at (0,-2.8);
\coordinate(dd1) at (-1.7,-.9);
\coordinate(dd2) at (0,-1.2);
\coordinate(dd3) at (0,-1.2);
\coordinate(dd4) at (1.7,-.9);

\path[fill, color=black] (source) circle (.6ex) 
node[above, font=\scriptsize]{$x^+$};
\draw[very thick, color=black, bend right=20] (sinkl) to (sinkr);
\node[align=center, font=\scriptsize] (Dminus) at (dminus)  {$D_X^-$};
\draw[thick, bend right=30] (source) to (sinkl);
\node[align=center, font=\scriptsize] (D0) at (d0)  {$D_X^0$};
\draw[thick] (source) to (sinkm);
\node[align=center, font=\scriptsize] (Dk) at (dk)  {$D_X^k$};
\draw[thick, bend left=30] (source) to (sinkr);
\node[align=center, font=\scriptsize] (Dr) at (dr)  {$D_X^r$};
\draw[thick, dotted, bend right=10] (dd1) to (dd2) node[]{};
\draw[thick, dotted, bend left=10] (dd4) to (dd3) node[]{};\end{tikzpicture}
\end{center}
The fixed points of the $\KK^*$-surface $X$
are the point $x^+ = \mathbf{z}(-) = [0,\ldots,0,1]$
and all the points of the curve $D_X^-$.
\end{proposition}  

\begin{proof}
Follow the lines of the proof of Proposition~\ref{prop-basic-geom-ee}.
\end{proof}

Let us shed some light on the meaning of the notations
(ee), (ep), $x_h$, $x^\pm$, $D_X^-$.
Whenever $\KK^*$ acts morphically on a normal projective variety
$X$, each point $x \in X$ gives rise to a morphism
$$
\mu_x \colon \PP_1 \ \to \ X,
$$
extending the orbit map $\KK^* \to X$,
$s \mapsto s \cdot x$. This allows us to define
the \emph{limit points} of the $\KK^*$-orbit
through $x$ as
$$
x_0
\ := \
\lim_{s \to 0} s \cdot x
\ := \
\mu_x(0),
\qquad\qquad
x_\infty
\ := \
\lim_{s \to \infty} s \cdot x
\ := \
\mu_x(\infty).
$$
Moreover, there is precisely one \emph{source} $F^+ \subseteq X$ and
precisely one \emph{sink} $F^- \subseteq X$:
these are components of the fixed point set such that
$$
X^+ \ := \ \{x \in X; \ x_0 \in F^+ \},
\qquad\qquad
X^- \ := \ \{x \in X; \ x_\infty \in F^-\}
$$
are open subsets of $X$; see~\cite[Thm.~9]{Kon}.
If $X$ is a surface, then we say that a fixed point $x \in X$ is
\begin{itemize}
\item[(e)]
\emph{elliptic}, if $\{x\}$ is the source or the sink,
\item[(p)]
\emph{parabolic}, if $x$ belongs to a curve consisting of fixed points,
\item[(h)]
\emph{hyperbolic}, if $x$ is neither elliptic nor parabolic.
\end{itemize}

\begin{remark}
Consider a $\KK^*$-surface $X = X(P)$.
Then, according to the type of~$X$, 
we directly obtain the following.
\begin{itemize}
\item[(ee)]
The fixed point $x_h \in X$ is hyperbolic,
$x^\pm \in X$ are elliptic with
$\{x^+\}$ being the source and $\{x^-\}$ the sink.
Moreover,
$$
X^+ \ = \ X \setminus D_X^{02},
\qquad\qquad
X^- \ = \ X \setminus D_X^{01}. 
$$
\item[(ep)]
The fixed point $x^+ \in X$ is elliptic with
$\{x^+\}$ being the source, $D_X^-$ consists
of parabolic fixed points and is the sink.
Moreover,
$$
X^+ \ = \ X \setminus D_X^-,
\qquad\qquad
X^- \ = \ X \setminus \{x^+\}. 
$$
\end{itemize}
\end{remark}

The key observation is that
Construction~\ref{constr:kstarsurf}
of the $\KK^*$-surface $X = X(P)$ delivers
us for free the divisor class group
$\Cl(X)$ and the Cox ring
$$
\mathcal{R}(X)
\ = \
\bigoplus_{\Cl(X)} \Gamma(X,\mathcal{O}(D)),
$$
where we refer to~\cite[Constr.~1.4.2.1]{ArDeHaLa}
for the precise definition of the Cox ring.
First we care about Weil divisors and the
divisor class group.

\begin{construction}
Consider $X=X(P)$ in $Z=Z(P)$ as provided
by Construction~\ref{constr:kstarsurf}.
Then we obtain homomorphisms of abelian
groups
$$
\ZZ^{r+2} \to \WDiv(Z), \quad a \mapsto D_Z(a),
\qquad 
\ZZ^{r+2} \to \WDiv(X), \quad a \mapsto D_X(a),
$$
by prescribing their values on the canonical
basis vectors of $\ZZ^{r+2}$ according to the
type of $P$ as follows:
$$
\begin{array}{llll}
\text{(ee)} \quad
&
\ZZ^{r+2} \to \WDiv(Z),
&
e_{0j} \mapsto D_Z^{0j},
&
e_i \mapsto D_Z^i,
\\[1ex]
&
\ZZ^{r+2} \to \WDiv(X),
&
e_{0j} \mapsto D_X^{0j},
&
e_i \mapsto D_X^i,
\\[2ex]
\text{(ep)} \quad
&
\ZZ^{r+2} \to \WDiv(Z),
&
e_i \mapsto D_Z^i,
&
e^- \mapsto D_Z^-,
\\[1ex]
&
\ZZ^{r+2} \to \WDiv(X),
&
e_i \mapsto D_X^i,
&
e^- \mapsto D_X^-.
\end{array}
$$
\end{construction}

\begin{proposition}
\label{prop:restrictdiv}
Consider $X=X(P)$ in $Z=Z(P)$ with
$K = \ZZ^{r+2}/P^*\ZZ^{r+1}$
and the projection $Q \colon \ZZ^{r+2} \to K$.
Then we have a commutative diagram
$$
\xymatrix{
{\WDiv^{\TT^n}(Z)}
\ar@/^2pc/[rrrr]^{\imath^*}
\ar[d]_{D \mapsto [D]}  
&&
{\ZZ^{r+2}}
\ar[ll]_{\quad D_Z(a) \, \mapsfrom \, a}
\ar[rr]^{a \, \mapsto \, D_X(a) \quad}
\ar[d]^Q
&&
{\WDiv^{\TT^n}(X)}
\ar[d]^{D \mapsto [D]}  
\\
{\Cl(Z)}
&&
K
\ar[ll]^{\cong}
\ar[rr]_{\cong}
&&
{\Cl(X)}
}
$$
where $\imath^*$ extends the pullback
of $\TT^n$-invariant Cartier divisors
to $\TT^n$-invariant Weil divisors.
In particular, for every character function
$\chi^u \in \KK(Z)$, we have 
$$
\div(\chi^u)
\ = \
D_Z(P^*u)
\ = \
D_X(P^*u)
\ = \ 
\div(\chi^u\vert_X).
$$
\end{proposition}

\begin{proof}
The commutative diagram is directly obtained
from~\cite[Prop.~2.1.2.7 and Prop.~3.2.5.4]{ArDeHaLa}.
Note that the toric ambient variety $Z$
of $X$ in the latter reference is an open toric
subvariety with complement of codimension at least two
in our $Z(P)$. In particular,~$Z$ and $Z(P)$ share the
same divisor class group.
\end{proof}

\begin{remark}
\label{rem:identifycl}
Given $X = X(P)$ in $Z = Z(P)$,
Proposition~\ref{prop:restrictdiv}
allows us to identify divisor
class groups and divisor classes
of $X$ and $Z$ via
$$ 
\Cl(Z) \ = \ K \ = \ \Cl(X),
\qquad\qquad
[D_Z(a)] \ = \ Q(a) \ = \ [D_X(a)].
$$
\end{remark}

\begin{proposition}
\label{prop:CartieronX}
Let $X=X(P)$ in $Z=Z(P)$ arise from
Construction~\ref{constr:kstarsurf}
and let $x \in X$.
Then, for any $a \in \ZZ^{r+2}$, the following
statements are equivalent.
\begin{enumerate}
\item
The divisor $D_X(a) \in \WDiv(X)$ is principal
near $x \in X$.
\item
The divisor $D_Z(a) \in \WDiv(Z)$ is principal
near $x \in Z$.  
\end{enumerate}
In particular, $X$ inherits $\QQ$-factoriality
from  $Z$, is of Picard number $\rho(X) = 1$
and for any $x \in X$, the local class group
$\Cl(X,x)$ equals $\Cl(Z,x)$.
\end{proposition}

\begin{proof}
A divisor is principal near a given point if
and only if it maps to zero in the corresponding
local class group.
Thus, the assertions follow from
Remark~\ref{rem:identifycl} and the descriptions
of the local class groups of $Z$ and $X$ given
in~\cite[Prop.~2.4.2.3 and~3.3.1.5]{ArDeHaLa}.
\end{proof}

The description of the Cox ring of $X = X(P)$
is now an immediate consequence of the
corresponding more general~\cite[Thm.~3.4.7.3]{ArDeHaLa}
on the case of torus actions of complexity one;
compare also~\cite[Prop.~10.9]{HaHe}
and~\cite[Cor.~3.9 and Constr.~6.13]{HaHiWr}.

\begin{proposition}
Consider $X=X(P)$ in $Z=Z(P)$ with
$K = \ZZ^{r+2}/P^*\ZZ^{r+1}$
and the projection $Q \colon \ZZ^{r+2} \to K$.
Then  the Cox ring of $X$ is given as 
$$
\mathcal{R}(X) \ = \ \mathcal{R}(Z) /  \bangle{g_2,\ldots,g_r},
$$
where the variables $T_{0j}$,~$T_i$ and $T^-$ are of degree
$Q(e_{0j})$, $Q(e_i)$ and $Q(e^-)$ in $\Cl(X) = K = \Cl(Z)$,
according to the type of $P$.
\end{proposition}

\begin{corollary}
\label{cor:non-toric}
Every $\KK^*$-surface arising from
Construction~\ref{constr:kstarsurf}
is non-toric.
\end{corollary}

\begin{proof}
On the one side, the Cox ring  of
any projective toric variety
is a polynomial ring~\cite{Co}.
On the other side, the Cox ring
$\mathcal{R}(X)$ of any $\KK^*$-surface
$X = X(P)$, has as its spectrum
$\bar X = V(g_2,\ldots, g_r) \subseteq \KK^{r+2}$, 
coming with a singularity at the origin.
Thus $\mathcal{R}(X)$ is not a polynomial ring.
\end{proof}

Let us see why Construction~\ref{constr:kstarsurf} delivers
in fact all non-toric, normal, rational, projective
$\KK^*$-surfaces of Picard number one.
We benefit from~\cite[Thm~5.4.1.5]{ArDeHaLa}, which similar
to corresponding statements on the more general constructions
from~\cites{ArDeHaLa,HaHe,HaHiWr},
finally relies on the description of the Cox ring of
a variety with torus action given in~\cite{HaSu}.

Recall that an \emph{isomorphism of $G$-varieties}
$X$ and $X'$, is a pair $(\varphi,\tilde \varphi)$, where
$\varphi \colon X \to X'$ is an isomorphism of
varieties and $\tilde \varphi \colon G \to G$
a isomorphism of algebraic groups such that
$\varphi(g \cdot x) = \tilde \varphi(g) \cdot \varphi(x)$
holds for all $g \in G$ and $x \in X$.

\begin{theorem}
\label{thm:X2XP}
Every non-toric, normal, rational, projective $\KK^*$-surface
of Picard number one is isomorphic to a $\KK^*$-surface
$X(P,\lambda)$ arising from Construction~\ref{constr:kstarsurf}.
\end{theorem}

\begin{proof}
Let $X$ be a non-toric, normal, rational, projective
$\KK^*$-surface of Picard number $\rho(X)=1$.
Then~\cite[Thm~5.4.1.5]{ArDeHaLa} tells us that $X$
is $\QQ$-factorial and that it is isomorphic to a
$\KK^*$-surface $X(A,P)$ with defining data $(A,P)$
as specified in \cite[Def~5.4.1.1 and Constr.~5.4.1.3]{ArDeHaLa}.
In the notation used there, the divisor class
group~$\Cl(X)$ is isomorphic to $\ZZ^{n+m}/P^*\ZZ^{r+1}$
and we obtain
$$
c(P) - r(P)
\ = \
n_0+\ldots+n_r+m - (r+1)
\ = \
\rk(\Cl(X))
\ = \
\rho(X)
\ = \
1
$$
for the numbers $c(P)$ of columns and $r(P)$ of rows of $P$,
using $\QQ$-factoriality of~$X$ to obtain the
third equality.
Thus, $P$ is the defining matrix of a fake weighted
projective space.
Moreover, the displayed formula forces $m \le 1$. Hence,
only the cases (e-e), (p-e) and (e-p)
of~\cite[Constr.~5.4.1.3]{ArDeHaLa} occur.
In the case (e-e) one of the~$n_i$ equals $2$,
all others equal $1$ and in the cases (p-e), (e-p)
all~$n_i$ equal $1$.

Now the admissible operations~\cite[Def.~6.3, no.~(i)-(vii)]{Hae}
allow us to bring $P$ into the desired shape.
First we remove all columns with $l_i=1$ via
operations no.~(vi) and~(i).
As $X$ is non-toric, we remain with $r \ge 2$.
Next we achieve $n_0=2$ by means of no.~(v), use no.~(iv)
to ensure $l_{01}/d_{01} > l_{02}/d_{02}$
in the case (e-e) by and turn (p-e) via no.~(ii) into (e-p)
in the other cases.
Then $P$ looks as wanted and we still have
$X \cong X(A,P)$ due to~\cite[Prop.~6.7]{Hae}.
Finally, applying no.~(vii), we can turn~$A$ into a shape
as in Remark~\ref{rem:XAP} and arrive at $X \cong X(P,\lambda)$.
\end{proof}

\section{Singularities of $\KK^*$-surfaces of Picard number one}
\label{sec:geom-kstar}

We discuss the possible singularities of
a $\KK^*$-surface $X = X(P)$.
We characterize quasismoothness, compute
the local Gorenstein indices and log
terminality.
As a new application, we observe that every
rational projective $\KK^*$-surface of
Picard number one with at most log
terminal singularities is necessarily
a del Pezzo surface; see
Corollary~\ref{cor:log2delpezzo}.

First note that by normality of $X = X(P)$,
all its singularities are isolated and thus
must be $\KK^*$-fixed points.
We call $x \in X$ \emph{quasi-smooth}
if $x = \pi(\hat x)$ holds with a smooth
point $\hat x \in \hat X$.
It turns out that the quasi-smooth points
of $X$ are precisely those which are at most
cyclic quotient singularities.

\goodbreak

\begin{proposition}
\label{prop:ee-smoothqsmooth}
Consider a $\KK^*$-surface $X = X(P)$
with $P$ of type (ee).
Then all points of $X$ different from
$x^+$, $x^-$, $x_h$ are smooth.
Moreover:
\begin{enumerate}
\item
$x^+$ is quasi-smooth if and only if $r=2$ and $l_{01}=1$,
\item
$x^+$ is smooth if and only if 
$r=2$, $l_{01}=1$ and $\det(v_{01},v_1,v_2) = \pm 1$,
\item
$x^-$ is quasi-smooth if and only if $r=2$ and $l_{02}=1$,
\item
$x^-$ is smooth if and only if 
$r=2$, $l_{02}=1$ and $\det(v_{02},v_1,v_2) = \pm 1$,
\item
$x_h$ is quasi-smooth,
\item
$x_h$ is smooth if and only if
$l_{01}d_{02}-l_{02}d_{01} = 1$.
\end{enumerate}
\end{proposition}

\begin{proof}
We show~(i). In order to see that
$x^+ = \mathbf{z}(02) = [0,1,0,\ldots,0]$ is quasi-smooth,
we take a look at the gradients of the homogenized
defining relations:
$$
\grad (g_j)
\ = \
(l_{01}T_{01}^{l_{01}-1}T_{02}^{l_{02}}, \, l_{02}T_{01}^{l_{01}}T_{02}^{l_{02}-1},
l_1T^{l_1-1}, 0, \ldots , 0, l_j T_j^{l_j-1}, \ldots, 0).
$$
Due to $l_j \ge 2$, evaluating at $x^+$ gives 
$(l_{01}T_{01}^{l_{01}-1}, 0, \ldots, 0)$ in all cases.
Thus, the Jacobian of $g_2,\ldots,g_r$ is of full rank at $x^+$
if and only if $r=2$ and $l_{01}=1$.

For Assertion~(ii), recall that $x^+$ is the (only)
point in the $\TT^n$-orbit of $Z = Z(P)$ corresponding
to the $\cone(v_{01},v_1,\ldots,v_r)$.
Thus,~\cite[Cor.~3.3.1.12]{ArDeHaLa} tells us that
$x^+$ is smooth in $X$ if and only if it
is quasi-smooth in $X$ and smooth in $Z$.
The latter two translate to $r=2$, $l_{01} = 1$
and $\det(v_{01},v_1,\ldots,v_r) = \pm 1$.
Assertions~(iii) and~(iv) are settled analogously.
For~(v), note that $x_h = [0,0,z_1,\ldots,z_r]$ with all
$z_j$ non-zero and thus the Jacobian of $g_2,\ldots,g_r$
always is of full rank at $x^+$.
For~(vi), note that $x_h$ lies in the $\TT^n$-orbit
of $Z$ corresponding to $\cone(v_{01},v_{02})$.
Again, \cite[Cor.~3.3.1.12]{ArDeHaLa} yields
the claim.
\end{proof}

For a $\KK^*$-surface $X = X(P)$ of type (ep), we denote
for $i = 0, \ldots, r$ by $x_i^- \in X$ the unique point
in the intersection of the curves  $D_X^i$ and $D_X^-$.

\begin{proposition}
\label{prop:ep-smoothqsmooth}
Consider a $\KK^*$-surface $X = X(P)$
of type (ep).
Then all points of $X$ different from $x^+$ and
$x_0^-, \ldots, x_r^-$ are smooth.
Moreover:
\begin{enumerate}
 \item
$x^+$ is not quasi-smooth,
\item
each of $x_0^-, \ldots, x_r^-$ is quasi-smooth
but not smooth.
\end{enumerate}
\end{proposition}

\begin{proof}
As in the preceding proof, we look at the Jacobian of
$g_2,\ldots,g_r$ and see that quasi-smoothness 
of $x^+$ would allow $l_i \ne 1$ at most twice, which
is not possible. This proves~(i).
Concerning the points of $D_X^-$, observe that each
point of $\bar X \cap V(T^-)$ is smooth in $\bar X$.
Thus, all points of $D_X^-$ are quasi-smooth.
Any $x \in D_X^-$ different from
$x_0^-, \ldots, x_r^-$ lies in the in
$\TT^n$-orbit of $Z(P)$ corresponding to $\cone(v_i)$
which consists of smooth points of $Z(P)$.
Moreover, $x_i^-$ lies in the $\TT^n$-orbit
of $Z(P)$ corresponding to $\cone(v_i,v^-)$,
which consists of singular points of $Z(P)$ due
to $l_i \ge 2$.
Thus, \cite[Cor.~3.3.1.12]{ArDeHaLa} yields
the assertions.
\end{proof}

For checking the del Pezzo property and computing
Gorenstein indices, we need an explicit description
of an (anti-)canonical divisor.
As our surfaces have a complete intersection
Cox ring, we can apply~\cite[Prop.~3.3.3.2]{ArDeHaLa}
to obtain the following.

\begin{proposition}
\label{prop:anticanonicalonX}
Let $X=X(P)$ in $Z=Z(P)$ arise from
Construction~\ref{constr:kstarsurf}.
Then, according to the type of $X$,
we obtain anticanonical divisors
on $X$ by
$$
\begin{array}{llcl}
\text{(ee)} \quad
&
-\mathcal{K}_X^0
&
=
&
D_X^{01} + D_X^{02} + D_X^1 + \ldots + D_X^r
- (r-1)(l_{01}D_X^{01} + l_{02}D_X^{02}),
\\[1ex]
&
-\mathcal{K}_X^i
&
=
&
D_X^{01} + D_X^{02} + D_X^1 + \ldots + D_X^r
-(r-1) l_iD_X^i,
\\[2ex]
\text{(ep)} \quad
&
-\mathcal{K}_X^i
&
=
&
D_X^0 + \ldots + D_X^r + D_X^- - (r-1)l_iD_X^i.
\end{array}
$$
\end{proposition}

\goodbreak

We express the local Gorenstein indices of
the fixed points of a $\KK^*$-surface~$X(P)$;
recall that the Gorenstein index of $X(P)$
is the least common multiple over the local
Gorenstein indices.
In a first step, we explicitly determine the defining
linear forms representing an anticanonical
divisor near the fixed points.

\begin{lemma}
\label{lem:uforms}
Let integers $l_0, \ldots, l_r > 0$ and
$d_0, \ldots, d_r$ be given and define
rational numbers
$$
m_i \ := \ \frac{d_i}{l_i},
\qquad\qquad
m \ := \ m_0+ \ldots + m_r.
$$
Assume that $m \ne 0$ holds and
consider the $r+1$ by $r+1$ matrix $B$
and $u_B\in \QQ^{r+1}$  given by 
$$
B
\ := \ 
\left[
\begin{array}{ccccc}
-l_0 & l_1 &        &      
\\
\vdots &     & \ddots  &   
\\
-l_0 &     &        &  l_r 
\\
d_0 & d_1 & \ldots &  d_r 
\end{array}                                     
\right],
\qquad
\begin{array}{lcl}
u_B & := & \frac{1}{m}(u_1,\ldots,u_r,\ell),
\\[2ex]
u_i & := & (r-1)m_i + \sum_{j \ne i} \frac{m_j}{l_i} - \frac{m_i}{l_j},
\\[2ex]
\ell & := & \frac{1}{l_0} + \ldots +  \frac{1}{l_r} - r +1.
\end{array}
$$
Then the linear form $u_B$ evaluates on the columns
$v_0,v_1,\ldots,v_r$ of the matrix $B$ as follows:
$$
\bangle{u_B,v_0} \ = \ 1 - (r-1)l_0,
\qquad
\bangle{u_B,v_i} \ = \ 1,
\quad i = 1, \ldots, r.
$$
\end{lemma}

\begin{proof}
One explicitly computes the evaluation of $u_B$ on each of the columns
of the matrix~$B$.
\end{proof}

\begin{proposition}
\label{prop:locgor-ee}
Let $X = X(P)$ be of type (ee) with 
$P =[v_{01},v_{01},v_1,\ldots,v_r]$.
Consider the matrices 
$$
B^+ = [v_{01},v_1,\ldots,v_r],
\qquad\qquad
B^- = [v_{02},v_1,\ldots,v_r],
$$
the associated linear forms $u^+ = u_{B^+}$ and $u^- = u_{B^-}$
according to Lemma~\ref{lem:uforms} and the linear form 
$$
u_h
\ = \
\frac{1}{l_{02}d_{01}-l_{01}d_{02}}
(d_{01}-d_{02}, 0, \ldots, 0, l_{01}-l_{02})
\ \in \
\QQ^{r+1}.
$$
Then the local Gorenstein indices of 
$x^+, x^-, x_h \in X$ are the unique positive integers
$\iota^+$, $\iota^-$, $\iota_h$ such that
$\iota^+u^+$, $\iota^-u^-$, $\iota_hu_h$ are primitive
integral vectors.
\end{proposition}

\begin{proof}
By definition, the local Gorenstein index of 
$x \in X$ is the minimal positive
integer $\iota(x)$ such that the $\iota(x)$-fold
of some anticanonical divisor of $X$ is Cartier
near~$x$.
For $x^+,x^-\in X$, consider $-\mathcal{K}_X^0$ from
Proposition~\ref{prop:anticanonicalonX}.
Propositions~\ref{prop:cartier-fwpp}
and~\ref{prop:CartieronX} guarantee
that any integral multiple
$-\iota^\pm \mathcal{K}_X^0$ is
Cartier near $x^\pm$ if and only if
it equals $\div(\chi^{\iota^\pm u^\pm})$
near $x^\pm$ with an integral linear form
$\iota^\pm u^\pm$.
This proves the assertion for $x^\pm$.
For $x_h$ we argue analogously,
using that $-\iota_h\mathcal{K}_X^1$
from Proposition~\ref{prop:anticanonicalonX}
equals $\div(\chi^{\iota_h u_h})$ near $x_h$,
provided that $\iota_h u_h$ is integral.
\end{proof}

\begin{proposition}
\label{prop:locgor-ep}
Let $X = X(P)$ be of type (ep) with  
$P =[v_0,v_1,\ldots,v_r,-e_{r+1}]$.
Consider the matrix
$$
B^- = [v_0,\ldots,v_r],
$$
the associated linear form $u^- = u_{B^-}$
as in Lemma~\ref{lem:uforms}
and, for $i=0,1,\ldots,r$, the linear forms
$$
u_0
\ = \
\frac{1}{l_0}\left(1-d_0, 0, \ldots, 0,-l_0 \right),
\qquad
u_i
\ = \
\frac{1}{l_i}\left(0, \ldots 0, d_i-1, 0, \ldots, 0,-l_i \right).
$$
Then the local Gorenstein indices of
$x^-,x_0,\ldots,x_r \in X$ are the unique positive integers
$\iota^-,\iota_0,\ldots,\iota_r$ such that
$\iota^-u^-,\iota_0u_0,\ldots,\iota_ru_r$
are primitive integral vectors.
\end{proposition}

\begin{proof}
Follow the lines of the proof of Proposition~\ref{prop:locgor-ee}.
\end{proof}

We characterize log terminality of rational
$\KK^*$-surfaces. 
Recall that for any normal surface $X$ with a $\QQ$-Cartier
canonical divisor $\mathcal{K}_X$ one considers
a resolution of singularities $\pi \colon X' \to X$
and the associated \emph{ramification formula}
$$
\mathcal{K}_X'
\ = \
\pi^* \mathcal{K}_X + \sum a(E) E,
$$
where $E$ runs through the exceptional
prime divisors and the $a(E) \in \QQ$
are the \emph{discrepancies}
of $\pi \colon X' \to X$.
Then a point $x \in X$ is called
\emph{log terminal} if we 
have $a_E > - 1$ for all exceptional
divisors $E$ contracting to $x$.

\begin{proposition}
\label{prop:logtermchar}
Let $X=X(P)$ in $Z=Z(P)$ arise from
Construction~\ref{constr:kstarsurf}.
Then, according to the type of $P$,
we have the following:
\begin{itemize}
\item[(ee)]
\begin{enumerate} 
\item
$x^+ \in X$ is log terminal if and only if
$\frac{1}{l_{01}} + \frac{1}{l_1} + \ldots + \frac{1}{l_r} > r-1$
\item
$x^- \in X$ is log terminal if and only if
$\frac{1}{l_{02}} + \frac{1}{l_1} + \ldots + \frac{1}{l_r} > r-1$
\end{enumerate}
\item[(ep)]
$x^- \in X$ is log terminal if and only if
$\frac{1}{l_0} + \ldots + \frac{1}{l_r} > r-1$.
\end{itemize}
\end{proposition}

\begin{proof}
The statements are special cases of more general
results; see~\cite[Cor.~8.12]{Hae} for the surface
case and~\cite[Cor.~4.6]{BeHaHuNi},~\cite[Thm.~3.13]{ArBrHaWr}
for higher dimensions.
\end{proof}

\begin{remark}
\label{rem:logtermchar}
A tuple $(q_0,\ldots, q_r)$
of positive integers
is called \emph{platonic}
if the following inequality is
satisfied:
$$
q_0^{-1} + \ldots + q_r^{-1} \ > \ r-1.
$$
If $q_0 \ge \ldots \ge q_r$ holds, then platonicity
of the tuple is equivalent to $q_3 = \ldots = q_r = 1$ and 
$(q_0,q_1,q_2)$ being one of
$$
(q_0,q_1,1),
\quad
(q_0,2,2),
\quad
(5,3,2),
\quad
(4,3,2),
\quad
(3,3,2).
$$
\end{remark}

%
The log terminal surface singularities are subdivided
into the types $A_n$, $D_n$, $E_6$, $E_7$ and $E_8$,
according to the possible shapes of their
(non-labelled) minimal resolution graphs

\def\AnDnEsix-res-graphs{
\begin{tikzpicture}[scale=0.3]
\sffamily
\node at (0,0) {$A_n$:};
\draw (2,0) circle (10pt);
\draw (4,0) circle (10pt);
\node[] (a1) at (2,0) {};
\node[] (a2) at (4,0) {};
\draw (8,0) circle (10pt);
\draw[] (a1) edge (a2);
\draw[dotted] (4.5,0) to (7.5,0);
\node at (-2+14,0) {$D_n$:};
\draw (0+14,1) circle (10pt);
\draw (0+14,-1) circle (10pt);
\node[] (ao) at (0+14,1) {};
\node[] (au) at (0+14,-1) {};
\draw (2+14,0) circle (10pt);
\draw (4+14,0) circle (10pt);
\node[] (a1) at (2+14,0) {};
\node[] (a2) at (4+14,0) {};
\draw (8+14,0) circle (10pt);
\draw[] (ao) edge (a1);
\draw[] (au) edge (a1);
\draw[] (a1) edge (a2);
\draw[dotted] (4.5+14,0) to (7.5+14,0);
\node at (-2+28,0) {$E_6$:};
\draw (0+28,0) circle (10pt);
\node[] (a0) at (0+28,0) {};
\draw (2+28,0) circle (10pt);
\node[] (a1) at (2+28,0) {};
\draw (4+28,0) circle (10pt);
\node[] (a2) at (4+28,0) {};
\draw (4+28,1.5) circle (10pt);
\node[] (a2o) at (4+28,1.5) {};
\draw (6+28,0) circle (10pt);
\node[] (a3) at (6+28,0) {};
\draw (8+28,0) circle (10pt);
\node[] (a4) at (8+28,0) {};
\draw[] (a0) edge (a1);
\draw[] (a1) edge (a2);
\draw[] (a2) edge (a3);
\draw[] (a2o) edge (a2);
\draw[] (a3) edge (a4);
\end{tikzpicture}
}

\def\Eseven-Eeight-res-graphs{
\begin{tikzpicture}[scale=0.3]
\sffamily
\node at (-2,0) {$E_7$:};
\draw (0,0) circle (10pt);
\node[] (a0) at (0,0) {};
\draw (2,0) circle (10pt);
\node[] (a1) at (2,0) {};
\draw (4,0) circle (10pt);
\node[] (a2) at (4,0) {};
\draw (4,1.5) circle (10pt);
\node[] (a2o) at (4,1.5) {};
\draw (6,0) circle (10pt);
\node[] (a3) at (6,0) {};
\draw (8,0) circle (10pt);
\node[] (a4) at (8,0) {};
\draw (10,0) circle (10pt);
\node[] (a5) at (10,0) {};
\draw[] (a0) edge (a1);
\draw[] (a1) edge (a2);
\draw[] (a2) edge (a3);
\draw[] (a2o) edge (a2);
\draw[] (a3) edge (a4);
\draw[] (a4) edge (a5);
\node at (-2+18,0) {$E_8$:};
\draw (0+18,0) circle (10pt);
\node[] (a0) at (0+18,0) {};
\draw (2+18,0) circle (10pt);
\node[] (a1) at (2+18,0) {};
\draw (4+18,0) circle (10pt);
\node[] (a2) at (4+18,0) {};
\draw (4+18,1.5) circle (10pt);
\node[] (a2o) at (4+18,1.5) {};
\draw (6+18,0) circle (10pt);
\node[] (a3) at (6+18,0) {};
\draw (8+18,0) circle (10pt);
\node[] (a4) at (8+18,0) {};
\draw (10+18,0) circle (10pt);
\node[] (a6) at (12+18,0) {};
\draw (12+18,0) circle (10pt);
\node[] (a5) at (10+18,0) {};
\draw[] (a0) edge (a1);
\draw[] (a1) edge (a2);
\draw[] (a2) edge (a3);
\draw[] (a2o) edge (a2);
\draw[] (a3) edge (a4);
\draw[] (a4) edge (a5);
\draw[] (a5) edge (a6);
\end{tikzpicture}
}

\begin{center}
  
\AnDnEsix-res-graphs

\medskip

\Eseven-Eeight-res-graphs

\end{center}

\noindent
where the vertices reflect the exceptional curves and
any edge joining two vertices indicates an intersection
of the corresponding exceptional curves in a (single)
common point.
%

\begin{proposition}
\label{prop:lconfigurations}
If a $\KK^*$-surface $X = X(P)$
is log terminal, then the possible tuples
$(l_{01},l_{02},l_1,\ldots,l_r)$
for type (ee)
and $(l_0,\ldots,l_r)$ for type (ep)
are the following:
$$
\setlength{\arraycolsep}{3pt}
\begin{array}{lllclll}
\text{\rm (eAeA):} \
&
(1,1,x_1,x_2),
&  
& \qquad &
\text{\rm (eAeD):} \
&
(1,y,2,2),
&
(1,2,y,2),
\\[2ex]
\text{\rm (eAeE):}
&
(1,z,3,2),
&
(1,3,z,2),
&&
\text{\rm (eDeD):}
&
(2,2,y,2),
&
(y_1,y_2,2,2),
\\[.5ex]
&
(1,2,z,3),
&
&&
&
(1,1,y,2,2),
&
\\[2ex]
\text{\rm (eDeE):}
&
(2,3,z,2),
&
(2,z,3,2),
&&
\text{\rm (eEeE):}
&
(2,2,z,3),
&
(z_1,z_2,3,2),
\\[.5ex]
&
&
&&
&
(3,3,z,2),
&             
(1,1,z,3,2),
\\[2ex]
\text{\rm (eDp):}
&
(y,2,2),
&
&&
\text{\rm (eEp):}
&
(z,3,2),
&
\end{array}
$$
where $2 \le x_1,x_2,y,y_1,y_2$,
we have $3 \le z,z_1,z_2 \le 5$
and the notation ``eA, eD, eE'' refers to log
terminal $x^\pm \in X$ of type $A_n$, $D_n$ or $E_6$, $E_7$, $E_8$.
\end{proposition}

\begin{proof}
Recall that for the type (ee) we have $l_1,\ldots,l_r \ge 2$
and for the type (ep) we have $l_0,\ldots,l_r \ge 2$.
Thus for a log terminal $X = X(P)$,
Proposition~\ref{prop:logtermchar} and Remark~\ref{rem:logtermchar}
force $r \le 3$ in the case (ee) and $r=2$ in the case (ep).
The possible configurations then result from the possible choices
of platonic triples, presented in Remark~\ref{rem:logtermchar}.

Proposition~5.1~(i) and~(iii) tell us that
``eA'' indicates precisely the quasismooth
elliptic fixed points.
These are in fact toric singularities,
see~\cite[Cor.~6.12]{HaHu}, hence of type $A_n$,
see~\cite[Thm.~10.2.3]{CoLiSc},
also~\cite[Prop.~3.8, Sum.~3.12]{Hae}.
For the non-toric cases, explicit descriptions
in terms of defining matrices are
given in~\cite[Ex.~4.8]{ArBrHaWr}
and~\cite[Prop.~8.4]{Hae}.
There we find exactly the platonic
triples associated with~``eD'' and ``eE''
in the assertion and, resolving 
singularities as in~\cite[Constr.~5.4.3.2]{ArDeHaLa}, 
we obtain the types
$D_n$ or $E_6$, $E_7$, $E_8$
for the corresponding resolution graphs.
\end{proof}

\begin{corollary}
\label{cor:log2delpezzo}
Every non-toric, rational, log terminal,
projective $\KK^*$-surface $X$ of Picard
number one is del Pezzo.
\end{corollary}

\begin{proof}
We have to show that $X$ has an ample
anticanonical divisor.
As $X$ is of Picard number one, it
suffices to find an effective
anticanonical divisor.
For~$X$ of type~(ee),
Proposition~\ref{prop:anticanonicalonX}
provides us with the anticanonical divisor
$$
-\mathcal{K}_X^0
\ = \ 
D_X^{01} + D_X^{02} + D_X^1 + \ldots + D_X^r
- (r-1)(l_{01}D_X^{01} + l_{02}D_X^{02}).
$$
In order to verify effectivity of $-\mathcal{K}_X^0$,
we have to show that the corresponding class in
$\Cl(X)_\QQ = \QQ \otimes_\ZZ \Cl(X) = \QQ$ is positive.
This is done by going through the cases
of Proposition~\ref{prop:lconfigurations}.
For instance, consider $(z_1,z_2,3,2)$ from (eEeE).
There we have
$$
z_1\omega_X^{01} + z_2\omega_X^{02} \ = \ 3 \omega_X^1,
\qquad\qquad
z_1\omega_X^{01} + z_2\omega_X^{02} \ = \ 2 \omega_X^2,
$$
reflecting the relations among the classes
$\omega_X^{0j}$, $\omega_X^i$ of
$D_X^{0j}$, $D_X^i$ in $\Cl(X)$ given by the
first two rows of the defining matrix $P$.
Together with $3 \le z_1,z_2 \le 5$ this allows
us to estimate the anticanonical class of $X$
as follows:
\begin{eqnarray*}
[-\mathcal{K}_X^0]
&  = & 
\omega_X^{01} + \omega_X^{02} + \omega_X^1 + \omega_X^2
- z_1\omega_X^{01} - z_2\omega_X^{02}
\\[.5em]
& = &  
\left(\frac{z_1}{3} + \frac{z_1}{2} + (1-z_1) \right) \omega_X^{01}
+
\left(\frac{z_2}{3} + \frac{z_2}{2} + (1-z_2) \right) \omega_X^{02}
\\[.5em]
& = &  
\left(\frac{5}{6} + \frac{1}{z_1} -1  \right) z_1 \omega_X^{01}
+
\left(\frac{5}{6} + \frac{1}{z_2} - 1 \right) z_2 \omega_X^{02}
\\[.5em]
&  > &
0,
\end{eqnarray*}
using the same notation for classes
in $\Cl(X)_\QQ$ and $\Cl(X)$.
The other cases of type~(ee) are settled by 
analogous computations.
For the type~(ep), we exemplarily treat~$(y,2,2)$.
With the anticanonical divisor $-\mathcal{K}_X^0$
from Proposition~\ref{prop:anticanonicalonX}, we obtain
\begin{eqnarray*}
[-\mathcal{K}_X^0]
&  = & 
\omega_X^{0} +  \omega_X^1  + \omega_X^2 + \omega_X^-
- y \omega_X^{0} 
\\[.5em]
& = &  
\omega_X^{0} +  \frac{y}{2} \omega_X^0  +  \frac{y}{2} \omega_X^0 + \omega_X^-   - y \omega_X^{0}
\\[.5em]
& = &  
\omega_X^{0} +  \omega_X^-
\\[.5em]
&  > &
0,
\end{eqnarray*}
where we made use of the relations $y \omega_X^{0} = 2 \omega_X^1$
and $y \omega_X^{0} = 2 \omega_X^1$ in $\Cl(X)$
given by the first two rows of the defining matrix~$P$.
The remaining case in type~(ep) is settled similarly.
\end{proof}

Note that any toric del Pezzo surface is log terminal
and, moreover, any Gorenstein rational del Pezzo surface
is log terminal as well~\cite[Thm.~3.4]{HiWa}.
Concerning higher Gorenstein indices, we have the
following.

\goodbreak

\begin{remark}
Corollary~\ref{cor:log2delpezzo} has no converse
in the sense that any rational del Pezzo
$\KK^*$-surface $X$ of Picard number one should
be log terminal. For instance,
$$
P_l
\ := \
\left[
\begin{array}{cccc}
-6 & -1 & l & 0 
\\
-6 & -1 & 0 & 2
\\
-7 & -1 & \frac{l+1}{2} & 1
\end{array}                            
\right],
\qquad
5 \ \le \ l \ \in \ 2\ZZ+1
$$
defines a del Pezzo $\KK^*$-surface $X_l = X(P_l)$ of
Picard number one and Gorenstein index two for
every $l$ in the range given above, but none of
the $X_l$ is log terminal.
\end{remark}

\section{Classifying $\KK^*$-surfaces of Picard number one}
\label{sec:classify-kstar}

We present our classification procedure for non-toric,
log terminal, rational, projective $\KK^*$-surfaces~$X$
of Picard number one and given Gorenstein index~$\iota$.
Recall from Corollary~\ref{cor:log2delpezzo} that all
these surfaces are in fact del Pezzo.
We provide explicit results for $1 \le \iota \le 200$,
summarized as follows.

\begin{theorem}
There are 154.138 families of non-toric, log terminal,
rational, projective $\KK^*$-surfaces of Picard number
one and Gorenstein index at most 200.
The numbers of families for given Gorenstein
index develop as follows:

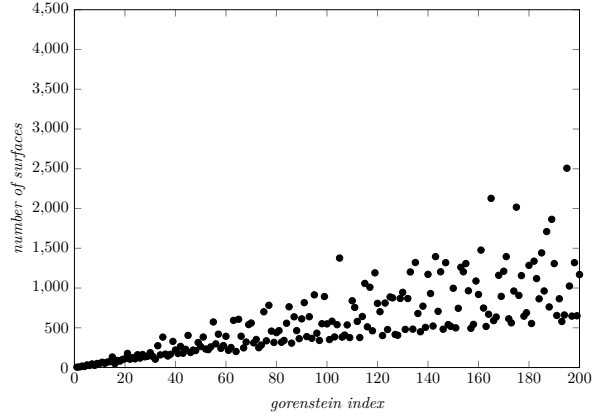
\begin{center}
  
\begin{tikzpicture}[scale=0.6]
\begin{axis}[
    xmin = 0, xmax = 200,
    ymin = 0, ymax = 4500,
    width = \textwidth,
    height = 0.75\textwidth,
    xtick distance = 20,
    ytick distance = 500,
    xlabel = {gorenstein index},
    ylabel = {number of surfaces}
]
\addplot[only marks] table {toricTable.txt};
\end{axis}
\end{tikzpicture}

\end{center}
\end{theorem}

The classification uses the presentation
$X = X(P) = X(P,\lambda)$
of the surfaces in question and is performed entirely
in terms of the defining matrices $P$.
The task is to derive suitably efficient bounds
on $P$ from the geometric properties of the
surface $X(P)$.
For this it turns out to be helpful that we
have a certain flexibility for the choice of
the defining data of a given surface.
Let us make this precise.

\begin{definition}
\label{def:valtrans}
Consider defining data $(P,\lambda)$
as in Construction~\ref{constr:kstarsurf}.
Then the  \emph{valid transformations}
on $(P,\lambda)$ are
\begin{itemize}[leftmargin=1.8cm]
\item[\hphantom{(ee)}\enspace $\addd_{a}$:]
given $a \in \ZZ^r$, add $(a,0) \cdot P$ to the last row of $P$, keep $\lambda$,
\item[\hphantom{(ee)}\enspace $\vt_1$:]
swap $l_1,l_2$, swap $d_1,d_2$, redefine
$\lambda := (1-\lambda_3, \ldots,1-\lambda_r)$,
\item[\hphantom{(ee)}\enspace $\vt_2$:]
swap $l_2,l_3$, swap $d_2,d_3$, redefine
$\lambda := (\lambda_3^{-1},\lambda_4\lambda_3^{-1}, \ldots,\lambda_r\lambda_3^{-1})$,
\item[\hphantom{(ee)}\enspace $\vt_i$:]
for $3 \le i \le r-1$ swap $l_i,l_{i+1}$, swap $d_i,d_{i+1}$, swap $\lambda_i, \lambda_{i+1}$,
\end{itemize}
and, according to the type of $P$, we allow
\begin{itemize}[leftmargin=1.8cm]
\item[(ee)\enspace $\flip$:]
swap the columns $v_{01}$, $v_{02}$, multiply the last
row of $P$ by $-1$,  keep $\lambda$,
\item[(ep)\enspace $\vt_0$:]
swap $l_0,l_1$ swap $d_0,d_1$, redefine
$\lambda := (\lambda_3^{-1}, \ldots, \lambda_r^{-1})$.
\end{itemize}
We write $(P,\lambda) \sim (P',\lambda')$ if $P$, $P'$ are of the
same type and size and if $(P',\lambda')$ arises from $(P,\lambda)$
by valid transformations; for $r=r'=2$ we also just write $P \sim P'$.
\end{definition}

\begin{remark}
\label{rem:valtrans1}
Let $(P,\lambda)$ be defining data as in
Construction~\ref{constr:kstarsurf}.
Then all valid transformations on $(P,\lambda)$
respect the conditions posed on $(P,\lambda)$ in
Construction~\ref{constr:kstarsurf}.
\end{remark}

\begin{proposition}
\label{prop:valtrans}
Two $\KK^*$-surfaces $X$ and $X'$ given by
defining data $(P,\lambda)$ and $(P',\lambda')$
are isomorphic to each other if and only if
$(P,\lambda) \sim (P',\lambda')$ holds.
\end{proposition}

\begin{proof}
First, note that each valid transformation
of Definition~\ref{def:valtrans}
is a composition of the more basic admissible
operations from~\cite[Def.~6.3]{Hae}.
Second, observe that any composition of
such admissible operations 
respecting the shape of a $(P,\lambda)$ of
Construction~\ref{constr:kstarsurf} is
a composition of valid transformations.
Thus, the assertion is a consequence
of~\cite[Prop.~6.7]{Hae}.
\end{proof}

For deciding whether or not two given defining data
as in Construction~\ref{constr:kstarsurf}
deliver isomorphic $\KK^*$-surfaces,
the following allows an efficient implementation.

\begin{definition}
\label{def:stdform}
Let $P$ be a defining matrix as in Construction~\ref{constr:kstarsurf}.
We say that~$P$ is in \emph{standard form}, if we have
$$
\begin{array}{llll}
\text{(ee)} \qquad
&
l_1 \geq \dots \geq l_r,
&
l_i = l_j \Rightarrow r(d_i,l_i) \ge r(d_j,l_j),
&
0 < d_i < l_i, \ i = 1, \dots, r,
\\[5pt]
\text{(ep)} \qquad
&
l_0 \geq \dots \geq l_r,
&
l_i = l_j \Rightarrow r(d_i,l_i) \geq r(d_j,l_j),
&
0 < d_i < l_i, \ i = 1, \dots, r,
\end{array}
$$
where, given $d \in \ZZ$ and $l \in \ZZ_{\ge 1}$, we denote by
$0 \le r(d,l) < l$ the remainder of $d$ modulo $l$. 
\end{definition}

\begin{remark}
\label{rem:valtrans2}
We can turn any defining matrix $P$ from~\ref{constr:kstarsurf}
via $\vt_i$ and $\add_a$ into standard form: first, suitable
$\vt_i$ establish the ordering of the $l_i$ and $r(d_i,l_i)$,
second $\add_a$ with $a_i := - \lfloor d_i/l_i \rfloor$ 
guarantees $0 < d_i < l_i$ for $i = 1,\ldots,r$.
\end{remark}

\begin{proposition}
\label{prop:valtrans2}
Let $(P,\lambda)$ and $(P',\lambda')$ be defining data
as in Construction~\ref{constr:kstarsurf} such that
$P$ and $P'$ are in standard form.
\begin{enumerate}
\item
If $P$ is of type (ep) and $(P,\lambda) \sim (P',\lambda')$ holds, then
we have $P = P'$.
\item
If $P$ is of type (ee) and $(P,\lambda) \sim (P',\lambda')$, 
then $P = P'$ or $P'$ arises from~$P$ via $\flip$ followed
by $\vt_i$ and $\add_a$ ensuring again standard form.
\item
Assume $(P,\lambda) \sim (P,\lambda')$. Then there is
a sequence of valid transformations $\vt_i$ leaving $P$
invariant that turns $\lambda$ into $\lambda'$.
\end{enumerate}
\end{proposition}

\begin{proof}
We show~(i). Standard form of $P$ and $P'$ implies $l_i = l_i'$ for
$i = 0, \dots, r$ and $d_i = r(d_i,l_i) = r(d_i',l_i') = d_i'$
for $i = 1, \dots, r$. This also forces $d_0 = d_0'$, hence $P = P'$.
For~(ii), note that $\flip$ commutes with $\vt_i$
and that $\add_a \circ \flip$ equals $\flip \circ \add_{-a}$.
Hence we may assume the sequence of valid transformations
turning~$P$ into $P'$ to start with at most one $\flip$ and contain
only $\vt_i$ and $\add_a$ otherwise.
Now similar arguments as used for~(i) lead to the two
cases of the assertion, where we arrive at the first one
if the sequence contains no $\flip$.
Assertion~(iii) is clear, as $\flip$ and $\add_a$
respect~$\lambda$.
\end{proof}

We turn to the classification.
A first step was done in
Proposition~\ref{prop:lconfigurations},
where we studied the effect of log terminality
on the possible choices of the entries
$l_{0j}$ and $l_i$ of $P$.
Basically, we go through the list given in
Proposition~\ref{prop:lconfigurations}
and provide bounds on~$P$ in each case.
We begin with the case (eAeA), comprising precisely
the quasi-smooth surfaces~$X(P)$.
This turns out to be by far the richest case.
In our treatment of this case, unit fractions
will play again a central role and pop up in a
similar way but more visibly as in~\cite{BaeHa},
where the Gorenstein threefold case is considered.

\begin{proposition}
\label{prop:qs-rho1-matrix}
Let $a,\iota^+,\iota^-,l_1,l_2,d_2 \in \ZZ_{>0}$
and $b \in \ZZ_{<0}$ such that
$l_1,l_2 \ge 2$ and $0 < d_2 < l_2$ hold.
Assume that the columns of the matrix
$$
\begin{array}{lcl}
P
& := &
\left[
\begin{array}{cccc}
-1 & -1 & l_1 & 0
\\
-1 & -1 & 0 &  l_2 
\\
0
&
-\frac{a\iota^+ - b\iota^-}{l_1l_2}
&
\frac{a\iota^+ - l_1d_2}{l_2}
&
d_2
\end{array}
\right]
\end{array}
$$
are pairwise distinct, integral and primitive.
Then $P$ defines a quasi-smooth
log del Pezzo $\KK^*$-surface $X = X(P)$
with $\rho(X) = 1$.
Assume in addition that
$$
\begin{array}{lcl}
\gamma^+
& := &
\left[
-\frac{a\iota^+-l_1d_2-l_2d_2}{al_2}, \
-\frac{a\iota^+-l_1d_2-l_2d_2}{al_2}, \
\frac{l_1+l_2}{a}
\right],
\\[1em]
\gamma^-
& := & 
\left[
-\frac{a\iota^+l_1+a\iota^+l_2-b\iota^-l_2-l_1^2d_2-l_1l_2d_2}{bl_1l_2}, \
\frac{b \iota^--l_1d_2-l_2d_2}{bl_2}, \ 
\frac{l_1+l_2}{b}
\right]
\end{array}
$$
are primitive vectors in $\ZZ^3$.
Then the local Gorenstein indices of
the elliptic fixed points $x^+,x^- \in X$
and the Gorenstein index of $X$ are given
by
$$
\iota^+ = \iota(x^+), \qquad 
\iota^- = \iota(x^-), \qquad 
\iota_X = \lcm(\iota^+,\iota^-).
$$
Finally, if $P$ and $\gamma^+,\gamma^-$ are
as above,
then there exist positive integers $a_1 \le a_2$
such that we have
$$
\frac{1}{\iota_X}
\ = \
\frac{1}{a_1l_1} + \frac{1}{a_2l_1} + \frac{1}{a_1l_2} + \frac{1}{a_2l_2} .
$$
\end{proposition}

\begin{lemma}
\label{lem:mplusmminus}
Consider any integral $(r+1) \times (r+2)$ matrix $P$ of the
following shape:
$$
P
\ = \
[v_{01},v_{02},v_1,\ldots,v_r]
\  = \
\left[
\begin{array}{ccccc}
-l_{01} & -l_{02} & l_1 &        & 
\\
\vdots & \vdots &     & \ddots
\\
-l_{01} & -l_{02} &     &        &  l_r 
\\
d_{01} & d_{02} & d_1 & \ldots &  d_r
\end{array}                                     
\right],
\quad
\begin{array}{c}
\scriptstyle
l_{01}, l_{02} \ge 1,
\\[1ex]
\scriptstyle
l_1,\ldots, l_r \ge 1,
\\[1ex]
\scriptstyle
\frac{d_{01}}{l_{01}} > \frac{d_{02}}{l_{02}}.
\end{array}
$$
Then the columns $v_{01},v_{02},v_1,\ldots,v_r$ of $P$
generate $\QQ^{r+1}$ as a convex cone if and only if
$$
m^-
\ := \ 
\frac{d_{02}}{l_{02}} + \frac{d_1}{l_1} + \ldots + \frac{d_r}{l_r}
\ < \ 0 \ < \ 
\frac{d_{01}}{l_{01}} + \frac{d_1}{l_1} + \ldots + \frac{d_r}{l_r}
\ =: \
m^+.
$$
\end{lemma}

\begin{proof}
As usual, denote $e_{r+1} := (0,\ldots,0,1) \in \QQ^{r+1}$.
Then we have the two positive combinations
$$
\frac{1}{l_{02}}v_{02} + \frac{1}{l_1}v_1 + \ldots + \frac{1}{l_r}v_r
\ = \
m^- e_{r+1},
\quad
\frac{1}{l_{01}}v_{01} + \frac{1}{l_1}v_1 + \ldots + \frac{1}{l_r}v_r
\ = \
m^+ e_{r+1}.
$$
So, if $m^-<0<m^+$, then $\pm e_{r+1}$ are positive
combinations over $v_{01},v_{02},v_1,\ldots,v_r$
and hence, by the shape of $P$, the latter ones
generate $\QQ^{r+1}$ as a convex cone.

Now assume that $v_{01},v_{02},v_1,\ldots,v_r$
generate $\QQ^{r+1}$ as a convex cone.
Observe that the nullspace of the upper $r \times (r+2)$
block of $P$ is generated by the coefficient
vectors of the above two positive combinations:
$$
\left(0,\frac{1}{l_{02}},\frac{1}{l_1},\ldots,\frac{1}{l_r}\right),
\qquad
\left(\frac{1}{l_{01}},0,\frac{1}{l_1},\ldots,\frac{1}{l_r}\right).
$$
In particular, any positive combination over
$v_{01},v_{02},v_1,\ldots,v_r$ representing 
$\pm e_{r+1}$ must be a positive combination
over $m^-e_{r+1}$ and $m^+e_{r+1}$.
Using $d_{02}/l_{02} < d_{01}/l_{01}$, we derive
the desired inequalities.
\end{proof}

\begin{proof}[Proof of Proposition~\ref{prop:qs-rho1-matrix}]
From the conditions on the matrix $P$
posed in Construction~\ref{constr:kstarsurf},
we only have to check that
the columns of $P$ generate $\QQ^3$ as cone.
Due to $b<0<a$, this holds by 
Lemma~\ref{lem:mplusmminus}.

Now, $X = X(P)$ is of Picard number one,
quasi-smooth and log terminal due to
Propositions~\ref{prop:CartieronX},~\ref{prop:ee-smoothqsmooth}
and~\ref{prop:logtermchar}.
With $u^+$ and $u^-$ from
Proposition~\ref{prop:locgor-ee},
we have
$$
\gamma^+ \ = \ \iota^+ u^+,
\qquad\qquad
\gamma^- \ = \ \iota^- u^-.
$$
Consequently, $\iota^\pm$ is the local Gorenstein index
of $x^\pm \in X$.
Moreover, Proposition~\ref{prop:locgor-ee}
shows $\iota(x_h) = 1$.
Proposition~\ref{prop:anticanonicalonX}
provides us with the anticanonical divisor
$$
-\mathcal{K}_X^0 \ = \ D_X^1 + D_X^2.
$$
Write $\iota := \iota_X$ for the Gorenstein index
of $X$. Then $-\iota \mathcal{K}_X^0$ is Cartier
near $x^-$, $x^+$ and the classes
$\omega_{0j}=(w_{0j},\eta_{0j})$
and
$\omega_i=(w_i,\eta_i)$ of
the $D_X^{0j}$ and $D_X^{i}$
satisfy
$$
a_1 w_{01} = \iota w_1 + \iota w_2,
\qquad\qquad
a_2 w_{02} = \iota w_1 + \iota w_2
$$
with $a_1,a_2 \in \ZZ_{\ge 1}$,
see Propositions~\ref{prop:cartier-fwpp}
and~\ref{prop:CartieronX}.
Moreover, the first two rows of $P$ encode
a relation among the $\omega_{0j}$ and $\omega_i$
such that altogether we obtain
$$
\left[
\begin{array}{cccc}
-1 & -1 & l_1 & 0
\\
-1 & -1 & 0 & l_2
\\
-a_1 & 0 & \iota & \iota
\\
0 & -a_2 & \iota & \iota   
\end{array}
\right]
\cdot
\left[
\begin{array}{c}
w_{01}
\\
w_{02}
\\
w_1
\\
w_2
\end{array}
\right]
\ = \
0.
$$
This identity implies that the above matrix has vanishing
determinant. A simple computation shows that the latter
is equivalent to
$$
\frac{1}{\iota}
\ = \
\frac{1}{a_1l_1} + \frac{1}{a_2l_1} + \frac{1}{a_1l_2} + \frac{1}{a_2l_2} .
$$
\end{proof}

\begin{proposition}
Let $X$ be a quasi-smooth, rational,
projective $\KK^*$-surface of Picard
number one.
Then $X \cong X(P)$ with $P$
given by $a,b,\iota^+,\iota^-,l_1,l_2,d_2$
as in Proposition~\ref{prop:qs-rho1-matrix}
satisfying all assumptions made there.
\end{proposition}

\begin{proof}
According to Theorem~\ref{thm:X2XP}, we may assume 
that $X=X(P)$ holds.
As $X$ is quasismooth, 
Proposition~\ref{prop:ee-smoothqsmooth}
tells us that $P$ must be of type (ee)
with $r=2$ and $l_{01} = l_{02} = 1$.
Thus, we can assume
$$
P
\ = \ 
\left[
\begin{array}{cccc}
-1 & -1 & l_1 & 0
\\
-1 & -1 & 0 &  l_2
\\
0
&
d_{02} 
&
d_1
&
d_2
\end{array}                 
\right],
\qquad
d_{02} < 0, \quad 0 < d_2 < l_2.
$$
Indeed, we achieve $d_{01}=0$ and $0 < d_2 < l_2$ by means
of suitable valid transformations on $P$.
Then the condition $d_{01}/l_{01} > d_{02}/l_{02}$
from Construction~\ref{constr:kstarsurf} turns into
$d_{02} < 0$.
Moreover, as the columns of $P$ generate $\QQ^3$ as a cone,
we have
$$
l_1l_2d_{02} + l_1d_1+l_2d_2 \ < \ 0 \ < \ l_2d_1+l_1d_2.
$$
The local Gorenstein index $\iota^+ := \iota(x^+)$
equals the order of the subgroup generated by
$\mathcal{K}_X$ in the local class group of~$x^+$. 
Thus, Lagrange's Theorem,
Corollary~\ref{cor:local-class-groups},
Proposition~\ref{prop:CartieronX}
and Lemma~\ref{lem:mplusmminus}
provide us with an
$a \in \ZZ_{>0}$ such that 
$$
a \iota^+
\ = \
\vert \Cl(X,x^+) \vert
\ = \
\vert \Cl(Z,\mathbf{z}(02)) \vert
\ = \ 
\det(v_{01},v_1,v_2)
\ = \
l_1d_2 + l_2d_1.
$$
Resolving for $d_2$, gives the desired entry
at the third place of the column $v_1$.
Using the same arguments for
$\iota^- := \iota(x^-)$, we gain $b \in \ZZ_{<0}$
providing the desired entry at the third place
of the column $v_{02}$ via 
$$
b \iota^-
\ = \
\vert \Cl(X,x^-) \vert
\ = \ 
\vert \Cl(Z,\mathbf{z}(01)) \vert
\ = \
\det(v_{02},v_1,v_2)
\ = \
l_1l_2d_{02} + a \iota^+ .
$$
Now one directly checks that the vectors $\gamma^\pm$
from Proposition~\ref{prop:qs-rho1-matrix}
represent the Cartier divisor
$-\iota^{\pm} \mathcal{K}_X$, where
$-\mathcal{K}_X =  D_X^{01} + D_X^{02}$,
near $x^\pm$ and thus are integral and primitive.
\end{proof}

\begin{remark}
\label{rem:unitfractions}
Consider a defining matrix $P$ as in   
Proposition~\ref{prop:qs-rho1-matrix}.
Then, for given~$\iota$, the presentation
of $1/\iota$ as a sum of four unit
fractions given there allows only finitely
many choices of the pairs $(l_1,l_2)$.
Moreover, we have $0 \le d_2 < l_2$
and all entries of~$P$ and $\gamma^+$, $\gamma^-$
are integers.
In particular,
$\gamma^\pm_3 \in \ZZ$ forces $a,b \mid l_1+l_2$.
Consequently, for fixed Gorenstein index
$\iota = \lcm(\iota^+,\iota^-)$,
all involved numbers $a,b,\iota^+,\iota^-,l_1,l_2,d_2$
are effectively bounded.
\end{remark}

\begin{algorithm}[Classification algorithm for non-toric,
quasi-smooth, rational, projective $\KK^*$-surfaces of Picard
number one]
\label{alg:classkstar}
\emph{Input:} A positive integer~$\iota$, the prospective
Gorenstein index.
\emph{Algorithm:}
\begin{itemize}
\item
open an empty list $S$ for defining matrices $P$;
\item
compute the list $L$ of all pairs $(l_1,l_2)$ of integers
$l_1,l_2 \ge 2$ that fit into the equation
$$
\frac{1}{\iota_X}
\ = \
\frac{1}{a_1l_1} + \frac{1}{a_2l_1} + \frac{1}{a_1l_2} + \frac{1}{a_2l_2} ;
$$
\item
for all $(l_1,l_2)$ from $L$ and all pairs $(\iota^+,\iota^-)$
of positive integers such that $\iota = \lcm(\iota^+,\iota^-)$ do
\begin{itemize}
\item[$\bullet$]
for all integers $a,d_2 > 0$ and $b < 0$ 
with $a,b \mid l_1+l_2$ build $P$ and $\gamma^+$, $\gamma^-$
as in Proposition~\ref{prop:qs-rho1-matrix};
\item[$\bullet$]
if $P$, $\gamma^+$, $\gamma^-$ satisfy all conditions imposed
in Proposition~\ref{prop:qs-rho1-matrix} and $P \not\sim P'$
for all $P'$ in $S$, then add $P$ to $S$;
\end{itemize}
\item
end do;
\end{itemize}
\emph{Output:} The list $S$.
Each $P$ from $S$ defines a non-toric, quasismooth,
rational, projective $\KK^*$-surface of Picard number
one and every $\KK^*$-surface with these properties
is isomorphic to $X(P)$ for precisely one $P$ from $S$. 
\end{algorithm}

We turn to the remaining cases, that means to the
non-quasi-smooth surfaces $X = X(P)$.
These cases provide considerably less examples
in each Gorenstein index.
We restrict ourselves to provide effective bounds
on the entries of the respective matrices~$P$
in each case, which allow a direct implementation
of feasible classification procedures.

\begin{proposition}
\label{prop:non-qs-bounds}
Let $X$ be a non-toric, non-quasi-smooth,
log terminal, rational, projective
$\KK^*$-surface of Picard number one.
Then $X \cong X(P)$ with a defining
matrix~$P$ taken from the following list:

\medskip
\noindent
Type (eAeD):
$$
\setlength{\arraycolsep}{3pt}
\begin{array}{ccc}
\left[
\begin{array}{rccc}
\scriptstyle
-1
&
\scriptstyle 
-\frac{b \iota^- + 4c \iota_h}{a \iota^+}
&
\scriptstyle 
2    
&
\scriptstyle 
0
\\
\scriptstyle 
-1
&
\scriptstyle 
-\frac{b \iota^- + 4c \iota_h}{a \iota^+}
&
\scriptstyle 
0
&
\scriptstyle 
2
\\
\scriptstyle 
0
&
\scriptstyle 
-c \iota_h
&
\scriptstyle 
\frac{a}{2} \iota^+ -1
&
\scriptstyle 
1
\end{array}
\right]
\renewcommand{\arraystretch}{0.6}
\begin{array}{l}
\scriptscriptstyle 
a = 1,2,4, 
\\
\scriptscriptstyle 
b = -2,-4, 
\\
\scriptscriptstyle 
0 < c,  
\\
\scriptscriptstyle 
c \mid (a\iota^+-b\iota^-),
\end{array}
&
&
\left[
\begin{array}{cccc}
\scriptstyle
-1
&
\scriptstyle
-2    
&
\scriptstyle
\frac{2a \iota^+ - b \iota^-}{2 \iota_h}
&
\scriptstyle
0
\\
\scriptstyle
-1
&
\scriptstyle
-2
&
\scriptstyle
0
&
\scriptstyle
2
\\
\scriptstyle
0
&
\scriptstyle
- \iota_h
&
\scriptstyle
\frac{a}{2}\iota^+ - \frac{2a\iota^+ - b \iota^-}{4 \iota_h}
&
\scriptstyle
1
\end{array}
\right]
\renewcommand{\arraystretch}{0.6}
\begin{array}{l}
\scriptscriptstyle 
b = -2,-4, 
\\
\scriptscriptstyle 
0 < a,  
\\
\scriptscriptstyle 
a \mid (4\iota_h-b\iota^-).
\end{array}                                                               
\end{array}                                                               
$$

\medskip
\noindent
Type (eAeE):
$$
\setlength{\arraycolsep}{3pt}
\begin{array}{ccc}
\left[
\begin{array}{rccc}
\scriptstyle
-1
&
\scriptstyle 
-z
&
\scriptstyle 
 3   
&
\scriptstyle 
0
\\
\scriptstyle 
-1
&
\scriptstyle 
-z
&
\scriptstyle 
0
&
\scriptstyle 
2
\\
\scriptstyle 
0
&
\scriptstyle 
\frac{b \iota^- - z a \iota^+}{6}
&
\scriptstyle 
\frac{a \iota^+ -3}{2} 
&
\scriptstyle 
1
\end{array}
\right]
\renewcommand{\arraystretch}{0.6}
\begin{array}{l}
\scriptscriptstyle 
z = 3,4,5, \, a = 1,5, 
\\
\scriptscriptstyle 
0 > b, \, b \mid (6-z), 
\\
\scriptscriptstyle 
\iota_h \mid \frac{b\iota^- - a\iota^+z}{6},
\end{array}
&
& 
\left[
\begin{array}{rccc}
\scriptstyle
-1
&
\scriptstyle 
-3
&
\scriptstyle 
 z   
&
\scriptstyle 
0
\\
\scriptstyle 
-1
&
\scriptstyle 
-3
&
\scriptstyle 
0
&
\scriptstyle 
2
\\
\scriptstyle 
0
&
\scriptstyle 
\frac{b \iota^- - 3a \iota^+}{2z}
&
\scriptstyle 
\frac{a \iota^+ - z}{2} 
&
\scriptstyle 
1
\end{array}
\right]
\renewcommand{\arraystretch}{0.6}
\begin{array}{l}
\scriptscriptstyle 
z = 4,5, 
\\
\scriptscriptstyle 
0 < a, \, a \mid(z+2), 
\\
\scriptscriptstyle 
0 > b, \, b \mid (6-z), 
\\
\scriptscriptstyle 
\iota_h \mid \frac{b\iota^- - 3a\iota^+}{2z},
\end{array}
\end{array}
$$
$$
\setlength{\arraycolsep}{3pt}
\begin{array}{c}
\left[
\begin{array}{rccc}
\scriptstyle
-1
&
\scriptstyle 
-2
&
\scriptstyle 
 z   
&
\scriptstyle 
0
\\
\scriptstyle 
-1
&
\scriptstyle 
-2
&
\scriptstyle 
0
&
\scriptstyle 
3
\\
\scriptstyle 
0
&
\scriptstyle 
\frac{b \iota^- - 2a\iota^+}{3z}
&
\scriptstyle 
\frac{a \iota^+ - dz}{3} 
&
\scriptstyle 
d
\end{array}
\right]
\qquad
\renewcommand{\arraystretch}{0.6}
\begin{array}{ll}
\scriptscriptstyle 
z = 3,4,5, \, d=1,2, 
&
\\
\scriptscriptstyle 
0 < a, \, a \mid (z+3), 
&
\scriptscriptstyle
\iota_h \mid \frac{b\iota^- - 2a\iota^+}{3z} .
\\
\scriptscriptstyle 
0 > b, \, b \mid (6-z),
&
\end{array}
\end{array}
$$

\goodbreak
\medskip
\noindent
Type (eDeD):
$$
\setlength{\arraycolsep}{3pt}
\begin{array}{ccc}
\left[
\begin{array}{rccc}
\scriptstyle
-2
&
\scriptstyle 
-2
&
\scriptstyle 
\frac{a \iota^+ - b \iota^-}{2c}    
&
\scriptstyle 
0
\\
\scriptstyle 
-2
&
\scriptstyle 
-2
&
\scriptstyle 
0
&
\scriptstyle 
2
\\
\scriptstyle 
-1
&
\scriptstyle 
-c-1
&
\scriptstyle 
\frac{a\iota^+}{4}
&
\scriptstyle 
1
\end{array}
\right]
\renewcommand{\arraystretch}{0.6}
\begin{array}{l}
\scriptscriptstyle 
\iota_h = 2, \, a = 2,4,
\\
\scriptscriptstyle 
b = -2,-4,
\\
\scriptscriptstyle 
0 < c,
\\
\scriptscriptstyle 
c \mid \frac{a\iota^+ - b\iota^-}{2},
\end{array}
& &
\left[
\begin{array}{rccc}
\scriptstyle
-y
&
\scriptstyle 
- \frac{yb\iota^- + 4c\iota_h}{a \iota^+}
&
\scriptstyle 
2 
&
\scriptstyle 
0
\\
\scriptstyle 
-y
&
\scriptstyle 
- \frac{yb\iota^- + 4c\iota_h}{a \iota^+}
&
\scriptstyle 
0
&
\scriptstyle 
2
\\
\scriptstyle 
\frac{a\iota^+}{4}-y
&
\scriptstyle 
\frac{b\iota^-}{4} - \frac{yb\iota^- + 4c\iota_h}{a \iota^+}
&
\scriptstyle 
1
&
\scriptstyle 
1
\end{array}
\right]
\renewcommand{\arraystretch}{0.6}
\begin{array}{l}
\scriptscriptstyle 
a = 2,4, \, b = -2,-4,
\\
\scriptscriptstyle 
0 < c,
\\
\scriptscriptstyle 
4c \mid (a\iota^+ - b\iota^-),
\\
\scriptscriptstyle 
y < \frac{a\iota^+ - 4c\iota_h}{b\iota^-},
\end{array}  
\end{array}
$$
$$
\left[
\begin{array}{ccccc}
\scriptstyle
-1
&
\scriptstyle 
-1
&
\scriptstyle 
y 
&
\scriptstyle 
0
&
\scriptstyle 
0
\\
\scriptstyle 
-1
&
\scriptstyle 
-1
&
\scriptstyle 
0
&
\scriptstyle 
2
&
\scriptstyle 
0
\\
\scriptstyle 
-1
&
\scriptstyle 
-1
&
\scriptstyle 
0
&
\scriptstyle 
0
&
\scriptstyle 
2
\\
\scriptstyle 
0
&
\scriptstyle 
- \frac{a\iota^+ + b\iota^-}{4y}
&
\scriptstyle 
\frac{a\iota^+ - 4y}{4}
&
\scriptstyle 
1
&
\scriptstyle 
1
\end{array}
\right]
\renewcommand{\arraystretch}{0.6}
\begin{array}{ll}
\scriptscriptstyle 
\iota_h=1,
&
\\
\scriptscriptstyle 
a = 2,4,
&
\scriptscriptstyle 
y \mid \frac{a\iota^+ + b\iota^-}{4}.
\\
\scriptscriptstyle 
b = 2,4,
&
\end{array}
$$

\goodbreak
\medskip
\noindent
Type (eDeE):
$$
\setlength{\arraycolsep}{3pt}
\begin{array}{ccc}
\left[
\begin{array}{rccc}
\scriptstyle
-2
&
\scriptstyle 
-3
&
\scriptstyle 
z    
&
\scriptstyle 
0
\\
\scriptstyle 
-2
&
\scriptstyle 
-3
&
\scriptstyle 
0
&
\scriptstyle 
2
\\
\scriptstyle 
-1
&
\scriptstyle 
- \frac{\iota_h-3}{2}
&
\scriptstyle 
\frac{a\iota^+}{4}
&
\scriptstyle 
1
\end{array}
\right]
\renewcommand{\arraystretch}{0.6}
\begin{array}{l}
\scriptscriptstyle 
z = 3,4,5,
\\
\scriptscriptstyle 
a = 2,4,
\end{array}
&
&
\left[
\begin{array}{rccc}
\scriptstyle
-2
&
\scriptstyle 
-z
&
\scriptstyle 
3    
&
\scriptstyle 
0
\\
\scriptstyle 
-2
&
\scriptstyle 
-z
&
\scriptstyle 
0
&
\scriptstyle 
2
\\
\scriptstyle 
-1
&
\scriptstyle 
-\frac{c\iota_h-z}{2}
&
\scriptstyle 
d
&
\scriptstyle 
1
\end{array}
\right]
\renewcommand{\arraystretch}{0.6}
\begin{array}{l}
\scriptscriptstyle 
z = 3,4,5, 
\\
\scriptscriptstyle 
d = \frac{\iota^+}{2}, \iota^+,
\\
\scriptscriptstyle 
0< c, \, c \mid (z-2).
\end{array}
\end{array}
$$

\goodbreak
\medskip
\noindent
Type (eEeE):
$$
\setlength{\arraycolsep}{2pt}
\begin{array}{cc}
\left[
\begin{array}{rccc}
\scriptstyle
-2 
&
\scriptstyle 
-2
&
\scriptstyle 
z  
&
\scriptstyle 
0
\\
\scriptstyle 
-2
&
\scriptstyle 
-2
&
\scriptstyle 
0
&
\scriptstyle 
3
\\
\scriptstyle 
-1
&
\scriptstyle 
-c-1
&
\scriptstyle 
\frac{a\iota^+ - 2dz + 3z}{6}
&
\scriptstyle 
d
\end{array}
\right]
\renewcommand{\arraystretch}{0.6}
\begin{array}{l}
\scriptscriptstyle 
z = 3,4,5, \, d = 1,2,
\\
\scriptscriptstyle 
0 < a, \, a \mid (6-z)
\\
\scriptscriptstyle 
\iota_h = 2, \, a \iota^+ < 3cz,
\\
\scriptscriptstyle 
3c < a \iota^+ +(6-z)\iota^-,
\end{array}
&   
\left[
\begin{array}{cccc}
\scriptstyle
-z_1
&
\scriptstyle 
-z_2
&
\scriptstyle 
3    
&
\scriptstyle 
0
\\
\scriptstyle 
-z_1
&
\scriptstyle 
-z_2
&
\scriptstyle 
0
&
\scriptstyle 
2
\\
\scriptstyle 
\frac{a\iota^+ - (d+3)z1}{6}
&
\scriptstyle 
\frac{b\iota^- - (d+3)z2}{6}
&
\scriptstyle 
d
&
\scriptstyle 
1
\end{array}
\right]
\renewcommand{\arraystretch}{0.6}
\begin{array}{l}
\scriptscriptstyle 
z_1,z_2 = 3,4,5, 
\\
\scriptscriptstyle 
d=1,2,
\\
\scriptscriptstyle 
0 < a, \, a \mid (6-z_1),
\\
\scriptscriptstyle 
0 > b, \, b \mid (6-z_2),
\end{array}
\end{array}
$$
$$
\begin{array}{c}  
\left[
\begin{array}{rccc}
\scriptstyle
-3
&
\scriptstyle 
-3
&
\scriptstyle 
z   
&
\scriptstyle 
0
\\
\scriptstyle 
-3
&
\scriptstyle 
-3
&
\scriptstyle 
0
&
\scriptstyle 
2
\\
\scriptstyle 
-d
&
\scriptstyle 
\frac{b \iota^- - a\iota^+ + 2dz}{2z}
&
\scriptstyle 
-\frac{(2d+3)z+a\iota^+}{6}
&
\scriptstyle 
1
\end{array}
\right]
\renewcommand{\arraystretch}{0.6}
\begin{array}{ll}
\scriptscriptstyle 
\iota_h = 3,
&
\scriptscriptstyle 
0 < a, \, a \mid (6-z),
\\
\scriptscriptstyle 
z = 4,5, \, d = 1,2,
&
\scriptscriptstyle 
0 > b, \, b \mid (6-z),
\end{array}
\end{array}
$$

$$
\left[
\begin{array}{ccccc}
\scriptstyle
-1
&
\scriptstyle 
-1
&
\scriptstyle 
z  
&
\scriptstyle 
0
&
\scriptstyle 
0
\\
\scriptstyle 
-1
&
\scriptstyle 
-1
&
\scriptstyle 
0
&
\scriptstyle 
3
&
\scriptstyle 
0
\\
\scriptstyle 
-1
&
\scriptstyle 
-1
&
\scriptstyle 
0
&
\scriptstyle 
0
&
\scriptstyle 
2
\\
\scriptstyle 
0
&
\scriptstyle 
- \frac{a\iota^+ + b\iota^-}{6z}
&
\scriptstyle 
\frac{a\iota^+ - 2zd - 3z}{6}
&
\scriptstyle 
d
&
\scriptstyle 
1
\end{array}
\right]
\renewcommand{\arraystretch}{0.6}
\begin{array}{ll}
\scriptscriptstyle 
\iota_h=1,
&
\\
\scriptscriptstyle 
z=3,4,5,
&
\scriptscriptstyle 
a > 0, \ a \mid (6-z),
\\
\scriptscriptstyle 
d=1,2,
&
\scriptscriptstyle 
b > 0, \ b \mid (6-z).
\end{array}
$$

\goodbreak
\medskip
\noindent
Types (eDp), (eEp):
$$
\setlength{\arraycolsep}{3pt}
\begin{array}{ccc}
\left[
\begin{array}{cccc}
\scriptstyle
-y
&
\scriptstyle 
2 
&
\scriptstyle 
0    
&
\scriptstyle 
0
\\
\scriptstyle 
-y
&
\scriptstyle 
0
&
\scriptstyle 
2
&
\scriptstyle 
0
\\
\scriptstyle 
\frac{a}{4}\iota^+-y
&
\scriptstyle 
1
&
\scriptstyle 
1
&
\scriptstyle 
-1
\end{array}
\right]
\renewcommand{\arraystretch}{0.6}
\begin{array}{l}
\scriptscriptstyle 
a=2,4,
\\
\scriptscriptstyle 
y \mid \frac{\iota_p(a\iota^+ + 4)}{4},
\end{array}
&
&  
\left[
\begin{array}{cccc}
\scriptstyle
-z
&
\scriptstyle 
3
&
\scriptstyle 
0    
&
\scriptstyle 
0
\\
\scriptstyle 
-z
&
\scriptstyle 
0
&
\scriptstyle 
2
&
\scriptstyle 
0
\\
\scriptstyle 
\frac{a\iota^+ - (2d+3)z}{6}
&
\scriptstyle 
d
&
\scriptstyle 
1
&
\scriptstyle 
-1
\end{array}
\right]
\renewcommand{\arraystretch}{0.6}
\begin{array}{l}
\scriptscriptstyle 
z=3,4,5,
\\
\scriptscriptstyle 
d=1,2,
\\
\scriptscriptstyle 
0 < a,
\\
\scriptscriptstyle 
a \mid (6-z).
\end{array}
\end{array}
$$

\noindent
Moreover, for fixed $\iota \in \ZZ_{>0}$,
the defining matrices $P$ from the above list
having integral primitive $\iota$-Gorenstein
forms give us all $\KK^*$-surfaces $X(P)$
of Gorenstein index $\iota$.
\end{proposition}

\begin{proof}
Let $X$ be a non-toric, non-quasi-smooth,
rational, projective $\KK^*$-surface
of Picard number one.
By Theorem~\ref{thm:X2XP}, we may assume $X = X(P)$ 
and Proposition~\ref{prop:lconfigurations}
gives us the possible configurations
of
$(l_{01},l_{02},l_1,\ldots,l_r)$
for type (ee)
and $(l_0,\ldots,l_r)$ for type (ep).
We exemplarily discuss the case
(1,1,y,2,2) from (eDeD). 
There, after suitable valid transformations,
we can assume
$$
P
\ = \ 
\left[
\begin{array}{rrrrr}
-1 & -1 & y & 0 & 0
\\
-1 & -1 & 0 & 2 & 0
\\
-1 & -1 & 0 & 0 & 2  
\\
0  & d_{02}  & d_1 & 1   & 1
\end{array}
\right],
\quad d_{02} < 0.
$$

Let $\iota^+, \iota^-, \iota_h$ be the
local Gorenstein indices of the
fixed points $x^+, x^-, x_h$, respectively.
Proposition~\ref{prop:locgor-ee}
shows $\iota(x_h) = 1$ and applying 
Corollary~\ref{cor:local-class-groups},
Proposition~\ref{prop:CartieronX}
and Lemma~\ref{lem:mplusmminus}
to $x^+$ yields an integer $a > 0$ such that
$$
a \iota^+
\ = \
\vert \Cl(X,x^+) \vert
\ = \
\vert \Cl(Z,\mathbf{z}(02)) \vert
\ = \
\det(v_{01},v_1,v_2,v_3)
\ = \
4y + 4d_1.
$$
Thus, we can replace the entry $d_1$ with 
$(a\iota^+ - 4y)/4$.
Applying again
Corollary~\ref{cor:local-class-groups},
Proposition~\ref{prop:CartieronX}
and Lemma~\ref{lem:mplusmminus},
now to $x^-$, yields an integer $b < 0$ with 
$$
b \iota^-
\ = \
-\vert \Cl(X,x^-) \vert
\ = \
-\vert \Cl(Z,\mathbf{z}(01)) \vert
\ = \
\det(v_{02},v_1,v_2,v_3)
\ = \
- a \iota^+ - 4yd_{02}.
$$
After replacing the entry $d_{02}$ with 
$-(a\iota^++b\iota^-)/(4y)$,
the matrix~$P$ looks as in the assertion.
The remaining task is to bound $a,b$ and $y$.
We compute
$$
\iota^+u^+
\ = \
{\textstyle
\left[
\frac{4}{a},
\
\frac{a \iota^+ - 4}{2a},
\
\frac{a \iota^+ - 4}{2a},
\
\frac{4}{a}
\right]
},
\qquad
\iota^-u^-
\ = \
{\textstyle
\left[
\frac{a \iota^+ + b \iota^- -4y}{by},
\
\frac{b \iota^- + 4}{2b},
\
\frac{b \iota^- + 4}{2b},
\
-\frac{4}{b}
\right]
},
$$
according to Proposition~\ref{prop:locgor-ee}.
As $\iota^+u^+$ and $\iota^-u^-$ are in particular
integral vectors, we arrive at the
bounding conditions
$$
a \mid 4,
\qquad
b \mid 4,
\qquad
4y \mid (a \iota^+ + b \iota^-) .
$$

\goodbreak

\noindent
Since $x^\pm$ both are of type $D$, their canonical
multiplicity $\zeta^\pm$ is one or two;
see~\cite[Def.~4.2 and Ex.~4.8]{ArBrHaWr}.
By definition, $\zeta^\pm$ equals the last component of
$\iota^\pm u^\pm$, which excludes $a=1$ and $b=-1$.
\end{proof}

\section{Tables for the classification}
\label{sec:tables}

As we have seen in Remark~\ref{rem:families}, log
del Pezzo surfaces with torus action may come in
families. We speak of a
\emph{sporadic isomorphy class} in case of a
zero dimensional family.

\begin{proposition}
\label{prop:toric-numbers}
There are precisely 117.065 isomorphy classes
of log del Pezzo surfaces of Picard number one and
Gorenstein index at most 200 that admit an effective
action of a two-dimensional torus.
All these isomorphy classes are sporadic and
the numbers $\mu_\iota$ of classes of Gorenstein
index $\iota$ are the following:
\begin{center} 
{\small
\begin{longtable}{c|cccccccccccc}
\(\iota\) & 1 & 2 & 3 & 4 & 5 & 6 & 7 & 8 & 9 & 10 & 11 & 12\\
    \(\mu_\iota\) & 5 & 7 & 18 & 13 & 33 & 26 & 45 & 27 & 51 & 51 & 67 & 53\\\hline
    \(\iota\) & 13 & 14 & 15 & 16 & 17 & 18 & 19 & 20 & 21 & 22 & 23 & 24\\
    \(\mu_\iota\) & 69 & 74 & 133 & 48 & 89 & 81 & 102 & 110 & 178 & 105 & 124 & 109\\\hline
    \(\iota\) & 25 & 26 & 27 & 28 & 29 & 30 & 31 & 32 & 33 & 34 & 35 & 36\\
    \(\mu_\iota\) & 161 & 119 & 164 & 135 & 142 & 187 & 140 & 105 & 274 & 159 & 383 & 169\\\hline
    \(\iota\) & 37 & 38 & 39 & 40 & 41 & 42 & 43 & 44 & 45 & 46 & 47 & 48\\
    \(\mu_\iota\) & 145 & 166 & 329 & 221 & 177 & 266 & 180 & 230 & 404 & 189 & 220 & 213\\\hline
    \(\iota\) & 49 & 50 & 51 & 52 & 53 & 54 & 55 & 56 & 57 & 58 & 59 & 60\\
    \(\mu_\iota\) & 315 & 264 & 384 & 233 & 225 & 260 & 573 & 298 & 420 & 241 & 276 & 393\\\hline
    \(\iota\) & 61 & 62 & 63 & 64 & 65 & 66 & 67 & 68 & 69 & 70 & 71 & 72\\
    \(\mu_\iota\) & 216 & 252 & 593 & 202 & 607 & 394 & 247 & 321 & 540 & 560 & 310 & 353\\\hline
    \(\iota\) & 73 & 74 & 75 & 76 & 77 & 78 & 79 & 80 & 81 & 82 & 83 & 84\\
    \(\mu_\iota\) & 249 & 283 & 701 & 336 & 783 & 458 & 316 & 439 & 464 & 318 & 341 & 557\\\hline
    \(\iota\) & 85 & 86 & 87 & 88 & 89 & 90 & 91 & 92 & 93 & 94 & 95 & 96\\
    \(\mu_\iota\) & 764 & 307 & 638 & 464 & 363 & 612 & 816 & 389 & 639 & 368 & 914 & 432\\\hline
    \(\iota\) & 97 & 98 & 99 & 100 & 101 & 102 & 103 & 104 & 105 & 106 & 107 & 108\\
    \(\mu_\iota\) & 341 & 551 & 893 & 549 & 352 & 583 & 385 & 539 & 1377 & 383 & 409 & 536\\\hline
    \(\iota\) & 109 & 110 & 111 & 112 & 113 & 114 & 115 & 116 & 117 & 118 & 119 & 120\\
    \(\mu_\iota\) & 377 & 840 & 756 & 580 & 377 & 642 & 1058 & 512 & 1010 & 462 & 1191 & 807\\\hline
    \(\iota\) & 121 & 122 & 123 & 124 & 125 & 126 & 127 & 128 & 129 & 130 & 131 & 132\\
    \(\mu_\iota\) & 702 & 402 & 811 & 478 & 888 & 876 & 416 & 406 & 869 & 946 & 480 & 868\\\hline
    \(\iota\) & 133 & 134 & 135 & 136 & 137 & 138 & 139 & 140 & 141 & 142 & 143 & 144\\
    \(\mu_\iota\) & 1202 & 483 & 1321 & 680 & 450 & 772 & 505 & 1172 & 931 & 522 & 1395 & 707\\\hline
    \(\iota\) & 145 & 146 & 147 & 148 & 149 & 150 & 151 & 152 & 153 & 154 & 155 & 156\\
    \(\mu_\iota\) & 1204 & 482 & 1319 & 540 & 518 & 997 & 499 & 745 & 1261 & 1204 & 1308 & 965\\\hline
    \(\iota\) & 157 & 158 & 159 & 160 & 161 & 162 & 163 & 164 & 165 & 166 & 167 & 168\\
    \(\mu_\iota\) & 493 & 543 & 1088 & 919 & 1477 & 748 & 517 & 670 & 2128 & 590 & 635 & 1160\\\hline
    \(\iota\) & 169 & 170 & 171 & 172 & 173 & 174 & 175 & 176 & 177 & 178 & 179 & 180\\
    \(\mu_\iota\) & 895 & 1211 & 1395 & 613 & 562 & 962 & 2017 & 907 & 1156 & 646 & 689 & 1285\\\hline
    \(\iota\) & 181 & 182 & 183 & 184 & 185 & 186 & 187 & 188 & 189 & 190 & 191 & 192\\
    \(\mu_\iota\) & 554 & 1338 & 1119 & 864 & 1442 & 963 & 1710 & 762 & 1864 & 1307 & 655 & 865\\\hline
    \(\iota\) & 193 & 194 & 195 & 196 & 197 & 198 & 199 & 200\\
    \(\mu_\iota\) & 579 & 661 & 2507 & 1025 & 647 & 1319 & 651 & 1169\\
\end{longtable}
}
\end{center}
\end{proposition}

\begin{proposition}
\label{prop:kstar-numbers}
There are 154.138 families of non-toric log del Pezzo surfaces
of Picard number one and Gorenstein index at most 200 that
admit an effective action of a one-dimensional torus.
Among those, one finds 152.018 sporadic isomorphy classes
and 2120 one-dimensional families.
Any two members of a family share the same Gorenstein index,
any two members stemming from different families are not
isomorphic to each other.
For the numbers $\mu_\iota$ of sporadic isomorphy classes
and~$\nu_\iota$ of one-dimensional families of Gorenstein
index $\iota$ we obtain
\begin{center}
{\small
\begin{longtable}{c|cccccccccccc}
    \(\iota\) & 1 & 2 & 3 & 4 & 5 & 6 & 7 & 8 & 9 & 10 & 11 & 12\\
    \(\mu_{\iota}\) & 12 & 10 & 33 & 24 & 75 & 37 & 96 & 54 & 102 & 71 & 168 & 83\\
    \(\nu_{\iota}\) & 1 & 0 & 3 & 1 & 5 & 0 & 4 & 2 & 7 & 0 & 8 & 2\\\hline
    \(\iota\) & 13 & 14 & 15 & 16 & 17 & 18 & 19 & 20 & 21 & 22 & 23 & 24\\
    \(\mu_{\iota}\) & 154 & 105 & 189 & 98 & 219 & 102 & 235 & 174 & 241 & 150 & 302 & 171\\
    \(\nu_{\iota}\) & 4 & 0 & 11 & 4 & 7 & 0 & 6 & 4 & 12 & 0 & 10 & 5\\\hline
    \(\iota\) & 25 & 26 & 27 & 28 & 29 & 30 & 31 & 32 & 33 & 34 & 35 & 36\\
    \(\mu_{\iota}\) & 261 & 149 & 323 & 218 & 385 & 192 & 303 & 211 & 370 & 207 & 567 & 225\\
    \(\nu_{\iota}\) & 8 & 0 & 13 & 6 & 10 & 0 & 6 & 5 & 11 & 0 & 25 & 5\\\hline
    \(\iota\) & 37 & 38 & 39 & 40 & 41 & 42 & 43 & 44 & 45 & 46 & 47 & 48\\
    \(\mu_{\iota}\) & 332 & 239 & 480 & 303 & 472 & 266 & 399 & 344 & 506 & 270 & 537 & 309\\
    \(\nu_{\iota}\) & 4 & 0 & 17 & 9 & 9 & 0 & 6 & 4 & 20 & 0 & 12 & 8\\\hline
    \(\iota\) & 49 & 50 & 51 & 52 & 53 & 54 & 55 & 56 & 57 & 58 & 59 & 60\\
    \(\mu_{\iota}\) & 488 & 277 & 556 & 346 & 523 & 334 & 753 & 418 & 502 & 328 & 831 & 486\\
    \(\nu_{\iota}\) & 9 & 0 & 14 & 8 & 9 & 0 & 23 & 9 & 14 & 0 & 15 & 7\\\hline
    \(\iota\) & 61 & 62 & 63 & 64 & 65 & 66 & 67 & 68 & 69 & 70 & 71 & 72\\
    \(\mu_{\iota}\) & 455 & 349 & 703 & 356 & 823 & 366 & 564 & 442 & 754 & 556 & 783 & 454\\
    \(\nu_{\iota}\) & 4 & 0 & 27 & 8 & 22 & 0 & 6 & 7 & 16 & 0 & 14 &
    10\\\hline
    \(\iota\) & 73 & 74 & 75 & 76 & 77 & 78 & 79 & 80 & 81 & 82 & 83 & 84\\
    \(\mu_{\iota}\) & 527 & 365 & 787 & 486 & 1021 & 482 & 724 & 579 & 666 & 410 & 979 & 632\\
    \(\nu_{\iota}\) & 4 & 0 & 24 & 8 & 25 & 0 & 10 & 13 & 17 & 0 & 14 &
    8\\\hline
    \(\iota\) & 85 & 86 & 87 & 88 & 89 & 90 & 91 & 92 & 93 & 94 & 95 & 96\\
    \(\mu_{\iota}\) & 864 & 383 & 881 & 600 & 984 & 537 & 937 & 579 & 736 & 549 & 1200 & 584\\
    \(\nu_{\iota}\) & 17 & 0 & 18 & 13 & 14 & 0 & 15 & 5 & 14 & 0 & 29 &
    12\\\hline
    \(\iota\) & 97 & 98 & 99 & 100 & 101 & 102 & 103 & 104 & 105 & 106 & 107 & 108\\
    \(\mu_{\iota}\) & 710 & 522 & 1076 & 588 & 822 & 535 & 872 & 775 & 1344 & 449 & 992 & 683\\
    \(\nu_{\iota}\) & 6 & 0 & 29 & 11 & 9 & 0 & 8 & 11 & 34 & 0 & 14 &
    10\\\hline
    \(\iota\) & 109 & 110 & 111 & 112 & 113 & 114 & 115 & 116 & 117 & 118 & 119 & 120\\
    \(\mu_{\iota}\) & 836 & 748 & 966 & 741 & 857 & 530 & 1226 & 737 & 1087 & 713 & 1882 & 900\\
    \(\nu_{\iota}\) & 8 & 0 & 22 & 17 & 9 & 0 & 24 & 6 & 28 & 0 & 37 &
    14\\\hline
    \(\iota\) & 121 & 122 & 123 & 124 & 125 & 126 & 127 & 128 & 129 & 130 & 131 & 132\\
    \(\mu_{\iota}\) & 827 & 450 & 1005 & 699 & 1511 & 731 & 833 & 695 & 1123 & 750 & 1206 & 965\\
    \(\nu_{\iota}\) & 11 & 0 & 16 & 9 & 20 & 0 & 8 & 11 & 18 & 0 & 14 &
    13\\\hline
    \(\iota\) & 133 & 134 & 135 & 136 & 137 & 138 & 139 & 140 & 141 & 142 & 143 & 144\\
    \(\mu_{\iota}\) & 1322 & 619 & 1485 & 809 & 1023 & 695 & 1239 & 1229 & 996 & 749 & 1817 & 821\\
    \(\nu_{\iota}\) & 15 & 0 & 40 & 14 & 9 & 0 & 12 & 13 & 15 & 0 & 36 &
    15\\\hline
    \(\iota\) & 145 & 146 & 147 & 148 & 149 & 150 & 151 & 152 & 153 & 154 & 155 & 156\\
    \(\mu_{\iota}\) & 1356 & 553 & 1313 & 777 & 1280 & 802 & 1039 & 918 & 1435 & 963 & 1606 & 1075\\
    \(\nu_{\iota}\) & 15 & 0 & 27 & 10 & 14 & 0 & 8 & 13 & 29 & 0 & 33 &
    9\\\hline
    \(\iota\) & 157 & 158 & 159 & 160 & 161 & 162 & 163 & 164 & 165 & 166 & 167 & 168\\
    \(\mu_{\iota}\) & 998 & 709 & 1516 & 1041 & 1594 & 694 & 1101 & 1010 & 1731 & 808 & 1771 & 1166\\
    \(\nu_{\iota}\) & 4 & 0 & 22 & 21 & 32 & 0 & 6 & 6 & 42 & 0 & 18 &
    19\\\hline
    \(\iota\) & 169 & 170 & 171 & 172 & 173 & 174 & 175 & 176 & 177 & 178 & 179 & 180\\
    \(\mu_{\iota}\) & 1160 & 922 & 1491 & 847 & 1231 & 878 & 1979 & 1142 & 1388 & 833 & 2074 & 1285\\
    \(\nu_{\iota}\) & 11 & 0 & 29 & 10 & 9 & 0 & 42 & 17 & 14 & 0 & 21 &
    17\\\hline
    \(\iota\) & 181 & 182 & 183 & 184 & 185 & 186 & 187 & 188 & 189 & 190 & 191 & 192\\
    \(\mu_{\iota}\) & 1007 & 971 & 1241 & 962 & 1587 & 848 & 1784 & 1207 & 2012 & 1162 & 1446 & 1030\\
    \(\nu_{\iota}\) & 8 & 0 & 19 & 12 & 24 & 0 & 24 & 7 & 42 & 0 & 16 &
    17\\\hline
    \(\iota\) & 193 & 194 & 195 & 196 & 197 & 198 & 199 & 200\\
    \(\mu_{\iota}\) & 1141 & 805 & 2246 & 1076 & 1525 & 1017 & 1375 & 1188\\
    \(\nu_{\iota}\) & 4 & 0 & 48 & 16 & 13 & 0 & 12 & 18\\\hline
\end{longtable}
} 
\end{center}
\end{proposition}

The following describes the possible isomorphies
inside the one-dimensional families of
Proposition~\ref{prop:kstar-numbers} and shows
in particular that each of these families
represents infinitely many isomorphy classes.

\begin{remark}
\label{rem:isoinfamily}
Inside the one-dimensional families occuring in
Theorem~\ref{prop:kstar-numbers},
it may happen that for fixed $P$, different
$\lambda = \lambda_3$ define isomorphic members.
According to Remark~\ref{rem:families}
and Propositions~\ref{prop:valtrans},~\ref{prop:valtrans2},~\ref{prop:non-qs-bounds}, this can
only happen in two settings.
The first one is
$$
P
\ = \ 
\left[
\begin{array}{ccccc}
\scriptstyle
-1
&
\scriptstyle 
-1
&
\scriptstyle 
y 
&
\scriptstyle 
0
&
\scriptstyle 
0
\\
\scriptstyle 
-1
&
\scriptstyle 
-1
&
\scriptstyle 
0
&
\scriptstyle 
2
&
\scriptstyle 
0
\\
\scriptstyle 
-1
&
\scriptstyle 
-1
&
\scriptstyle 
0
&
\scriptstyle 
0
&
\scriptstyle 
2
\\
\scriptstyle 
0
&
\scriptstyle 
- \frac{a\iota^+ + b\iota^-}{4y}
&
\scriptstyle 
\frac{a\iota^+ - 4y}{4}
&
\scriptstyle 
1
&
\scriptstyle 
1
\end{array}
\right]
\renewcommand{\arraystretch}{0.6}
\begin{array}{ll}
\scriptscriptstyle 
\iota_h=1,
&
\\
\scriptscriptstyle 
a = 2,4,
&
\scriptscriptstyle 
y \mid \frac{a\iota^+ + b\iota^-}{4}.
\\
\scriptscriptstyle 
b = 2,4.
&
\end{array}
$$
Here, $\lambda' = \lambda, \lambda^{-1}$ give
$X(P,\lambda') \cong X(P,\lambda)$.
More isomorphies inside the family are only
found for $y=2$ and $d_1=1$,
forcing $a \iota^+ = 12$ and $8 \mid 12+b\iota^-$.
In this case, 
$$
\lambda'
\ = \
\lambda, \ 1 - \lambda, \ \frac{1}{\lambda}, \ \frac{1}{1-\lambda}, \ \frac{\lambda-1}{\lambda}, \ \frac{\lambda}{\lambda-1}
$$
are precisely the values with $X(P,\lambda') \cong X(P,\lambda)$;
observe that the possible $\lambda'$ form the orbit through
$\lambda$ under an action of the symmetric group $S_3$ with
$\vt_1$, $\vt_2$ and $\vt_1 \circ \vt_1$ representing the transpositions.
The second setting is 
$$
P \ = \ 
\left[
\begin{array}{ccccc}
\scriptstyle
-1
&
\scriptstyle 
-1
&
\scriptstyle 
z  
&
\scriptstyle 
0
&
\scriptstyle 
0
\\
\scriptstyle 
-1
&
\scriptstyle 
-1
&
\scriptstyle 
0
&
\scriptstyle 
3
&
\scriptstyle 
0
\\
\scriptstyle 
-1
&
\scriptstyle 
-1
&
\scriptstyle 
0
&
\scriptstyle 
0
&
\scriptstyle 
2
\\
\scriptstyle 
0
&
\scriptstyle 
- \frac{a\iota^+ + b\iota^-}{6z}
&
\scriptstyle 
\frac{a\iota^+ - 2zd - 3z}{6}
&
\scriptstyle 
d
&
\scriptstyle 
1
\end{array}
\right]
\renewcommand{\arraystretch}{0.6}
\begin{array}{ll}
\scriptscriptstyle 
\iota_h=1,
&
\\
\scriptscriptstyle 
z=3,4,5,
&
\scriptscriptstyle 
a > 0, \ a \mid (6-z),
\\
\scriptscriptstyle 
d=1,2,
&
\scriptscriptstyle 
b > 0, \ b \mid (6-z).
\end{array}
$$
Non-trivial isomorphies inside such a family can only occur
if we have $z=3$ and $a\iota^+ = 12d+9$. In this
case, $\lambda' = \lambda, 1-\lambda$ are exactly
the values with $X(P,\lambda') \cong X(P,\lambda)$.
\end{remark}

\end{document}